\documentclass[british]{article}
\usepackage{ae,aecompl}
\usepackage[T1]{fontenc}
\usepackage[latin9]{inputenc}
\usepackage{geometry}
\usepackage{amsmath}
\usepackage{amsthm}
\usepackage{amssymb}
\usepackage{esint}
\usepackage{color}
\usepackage{hyperref}
\usepackage{mathtools}
\makeatletter
\theoremstyle{plain}
\newtheorem{thm}{\protect\theoremname}
\theoremstyle{definition}
\newtheorem{defn}[thm]{\protect\definitionname}
\theoremstyle{plain}
\newtheorem{prop}[thm]{\protect\propositionname}
\theoremstyle{plain}
\newtheorem{cor}[thm]{\protect\corollaryname}
\theoremstyle{remark}
\newtheorem{rem}[thm]{\protect\remarkname}
\theoremstyle{plain}
\newtheorem{lem}[thm]{\protect\lemmaname}

\@ifundefined{date}{}{\date{}}

\usepackage{bbm}
\usepackage{babel}

\providecommand{\corollaryname}{Corollary}
\providecommand{\definitionname}{Definition}
\providecommand{\lemmaname}{Lemma}
\providecommand{\propositionname}{Proposition}
\providecommand{\remarkname}{Remark}
\providecommand{\theoremname}{Theorem}

\newcommand{\dto}{\xrightarrow[N\rightarrow\infty]{\textup{d}}}
\newcommand{\ed}[1]{\frac{1}{#1}}
\newcommand{\supp}{\text{supp}}

\newcommand{\bJ}{\mathbf{J}}
\newcommand{\bI}{\mathbf{I}}
\newcommand{\jj}{\bar{J}}

\begin{document}
\global\long\def\IN{\mathbb{N}}%
\global\long\def\II{\mathbbm{1}}%
\global\long\def\IZ{\mathbb{Z}}%
\global\long\def\IQ{\mathbb{Q}}%
\global\long\def\IR{\mathbb{R}}%
\global\long\def\IC{\mathbb{C}}%
\global\long\def\IP{\mathbb{P}}%
\global\long\def\IE{\mathbb{E}}%

\title{Optimal Weights in a Two-Tier Voting System\\ with Mean-Field Voters}
\author{Werner Kirsch\thanks{FernUniversit\"at in Hagen, Germany, werner.kirsch@fernuni-hagen.de}
\;and Gabor Toth\thanks{IIMAS-UNAM, Mexico City, Mexico, gabor.toth@iimas.unam.mx}}
\maketitle
\begin{abstract}
\noindent We analyse two-tier voting systems with voters described by a multi-group
mean-field model that allows for correlated voters both within groups
as well as across group boundaries. In this model voters are influenced by voters within their group (constituency, member state, etc.) in a positive way. Across group boundaries positive or negative influence is considered.

\noindent The objective is to determine
the optimal weights each group receives in the council, the upper level of the voting system, to minimise
the expected quadratic deviation of the council vote from a hypothetical
referendum of the overall population in the large population limit. The mean-field model exhibits
different behaviour depending on the intensity of interactions between
voters. When interaction is weak, we obtain optimal weights given
by the sum of a constant term and a term proportional
to the square root of the group's population. When interaction is
strong, the optimal weights are in general not uniquely determined.
Indeed, when all groups are positively coupled, any assignation of
weights is optimal. For two competing clusters of groups, the difference
in total weights must be a specific number, but the assignation of
weights within each cluster is arbitrary. We also obtain conditions
for both interaction regimes under which it is impossible to reach
the minimal democracy deficit as some of the weights may be negative.
\end{abstract}
Keywords: two-tier voting systems, probabilistic voting, mean-field
models, democracy deficit, optimal voting weights

2020 Mathematics Subject Classification: 91B12, 91B14, 82B20

\section{\label{sec:Introduction}Introduction}

This article studies yes-no-voting in two-tier voting systems. In
two-tier voting systems, the overall population is subdivided into
$M\in\IN$ groups (such as the member states of the European Union) of population size $N_{\lambda}\in \IN$ for each group $\lambda=1,\ldots,M $. Each
group sends a representative to a council which makes decisions for
the union. The representatives cast their vote (`aye' or `nay')
according to the majority in their respective group. For groups of
different sizes, it is natural to assign different voting weights
to the representatives.

These weights are fixed at a constitutional design stage prior to the voting in day-to-day decision making. The weights purposely structure future voting processes behind a `veil of ignorance,' like the respective voting provisions in the UN Charter, the Treaty on (the Functioning of the) European Union, the Articles of Agreement of the International Monetary Fund, etc. The respective constitutional arrangement may specify different weights and quotas for different types of decisions (e.g., all members of the UN Security Council have identical weight for procedural decisions but not for substantive decisions) or for different policy domains (all EU member states have identical weights in `sensitive' areas such as taxation and foreign policy, but cast population-dependent weights on proposals concerning the single market, agriculture, etc.). In any case, a given -- decision-type or domain-contingent -- vector of weights applies to a potentially very long sequence of `aye' or `nay' decisions on yet unknown proposals. It should not be chosen arbitrarily but `optimally' following
a fixed objective we shall discuss below.

In representative democracy, it is always an objective for a constitutional design process to reproduce the decisions of a hypothetical referendum in the decision of the legislative body. In the case of a two-tier system, this means to ensure that the decisions of the council reflect the will of the citizens. There is no way to choose the weights in the council such that the council \emph{always} agrees with the popular vote. The best one can do is to make sure it does \emph{most of the time}, a term we are trying to make precise below; in fact, this is one of the main purposes of this paper. 

One way to approach this question is to look at the power index of a voter in one of the constituencies, i.e.\! the (indirect) influence this voter has on the decisions of the council.
In a `fair' voting system, the influence of a voter should be independent of the voter's home state. This approach was introduced by Penrose \cite{Penrose} who used what is now known as the Banzhaf power index or Penrose-Banzhaf power index (see \cite{Banzhaf,Fel1999,Fel2005}). This approach leads to the famous square root law which states that fair voting weights should be proportional to $\sqrt{N_{\lambda}}$ for each group $\lambda=1,\ldots,M $.

If one instead defines `fair representation' in terms of the Shapley-Shubik index (see \cite{ShapShub1954}), then optimal voting weights should be proportional to $N_{\lambda}$. The difference between these two indices comes from a different `counting' of coalitions of voters which is equivalent to assigning a certain probability to each voting outcome.

A second path to optimal weights was opened by Felsenthal and Machover \cite{Fel1999}. These authors determined optimal weights in such a way that the \emph{democracy deficit}, i.e.\! the `expected' difference between the council vote and a hypothetical referendum among all voters, is as small as possible in a sense we shall make precise below. As the term `expected' suggests, this approach requires some sort of probability behind the voting behaviour. The proposals the voters cast their votes on in the future are completely unpredictable (i.e.\! `random') during the constitutional process, but some voting outcomes may seem to be more likely than others. Felsenthal and Machover assume that the voters react independently of each other to the randomly selected proposal (with `aye' and `nay' equally likely). This assumption leads to the Penrose-Banzhaf power index and the square root law for the optimal weights.

Straffin \cite{Straffin} considers a probability distribution on voting outcomes which leads instead to the Shapley-Shubik power index. We call this distribution the Shapley-Shubik distribution. Under the Shapley-Shubik distribution, the probability that exactly $k$ out of $N$ voters vote `aye' equals $1/(N+1)$ independently of $k $, and for a given $k $ all coalitions with $k $ members have the same probability. Note that this makes the votes dependent on each other.

A probabilistic model behind a constitutional process, be it the independence assumption of \cite{Fel1999}, the distribution in \cite{Straffin}, or any other probability distribution implicitly assumes that the correlation between voters remains more or less constant over time. Moreover, said correlation measures the degree of dependence of voters on each others' decisions taken over all proposals within the whole space of proposals or within the specified area of proposals.

The model of Felsenthal and Machover \cite{Fel1999} has been extended in various directions. Barber\`a and Jackson \cite{Barbera} developed a model in which the total utility is maximised rather than the bare yes-no-voting. Koriyama, Laslier, Mac\'{e}, and Treibich \cite{Kor} investigate the effect a positive correlation among voters has in terms of degressive proportionality. Kurz, Maaser, and Napel \cite{KMN2017} extend the space of yes-no-decisions to decisions reflected by a real number (e.g.\! a budget limit).

The paper \cite{SRL} emphasises the role of correlations between voters and thus of the choice of a voting measure to determine
optimal weights which minimise the democracy deficit. There
two families of voting measures were introduced. The first one, the `collective bias model,' generalises the Shapley-Shubik measure. In this model, there is some common belief (a system of common values or a dominating group of opinion makers) inside a constituency, which influences the voting behaviour of the entire constituency.

The second model introduced in \cite{SRL} is the `mean-field model' (MFM), which is borrowed from the statistical physics of magnetism. In physics, the model was introduced to describe a collection of small magnets (`spins') which have a tendency to align. The analogue in voting theory is that the voters inside a group influence each other so that a tendency to vote alike arises, i.e.\! there is collective behaviour inside the constituency.

The key characteristic of the MFM is that it exhibits what physicists refer to as a `phase transition,' i.e.\! a sudden qualitative change in the voting behaviour. At a certain threshold of the values of the parameter(s) which characterises the correlation between voters, the cohesion of the population in terms of their voting behaviour changes abruptly. For small parameter values, which stand for weak interactions between voters, there is a weak correlation between votes within each group. This weak correlation manifests in the form of small (or microscopic) majorities, typically close to a tie. As the parameter value increases, the correlation becomes slightly stronger, until at a certain critical value the typical magnitude of the majority jumps suddenly.

Under the models studied in \cite{SRL}, voting behaviour in the same group is no longer independent while voting results from different groups still are. The first studies of voting systems with voters' dependencies across group borders were made in \cite{Langner,KL1}, and \cite{KT2021_CBM}. In the present paper we extend work on the MFM from \cite{Toth} and
treat
a class of voting measures
which extends the impartial culture (see e.g.\! \cite{Gui1952,GK1968,GL2017,KMN2021}) by allowing correlations both
between voters in the same group as well as correlations across group
borders. The MFM has been extensively studied in physics and applied to the social
sciences. Models from statistical mechanics were first used by F\"ollmer
\cite{F=0000F6} to study social interactions. The MFM specifically
was first employed in \cite{BD}. See \cite{CGh,GBC,OEA,LSV}
for other applications.

In this article, the notion of `fair voting weights' corresponds to the `optimal voting weights' that minimise the democracy deficit (cf. Definition \ref{def:democracy-deficit}). Instead of `constituencies' we will refer to `groups' of voters.  The intuitive notions of interactions between voters or the cohesion of a group or the entire population in the MFM will be referred to as `coupling,' a term made formal in the definition of the model in Section \ref{subsec:Model}.

In the present paper, we shall prove that asymptotically the optimal weights for the MFM \emph{with interacting groups} are proportional to
   $\sqrt{N_{\lambda}}~+~C$
as long as the coupling is not too strong. The constant $C$ in the above expression reflects the influence from other groups, whereas the term $\sqrt{N_{\lambda}} $ comes from the coupling within the group $\lambda $.

If the coupling is strong, then, under certain assumptions on the form of the coupling, we find that the optimal weights are essentially arbitrary, i.e.\!, for large $N_{\lambda}$, the democracy deficit is asymptotically independent of the choice of the voting weights. This can be explained by the fact that under strong coupling most voters will agree anyway, so it does not matter how we weight the different groups.

The rest of the paper is organised as follows: in Section \ref{sec:Def}, we first define some basic concepts such as voting measures, the democracy deficit, and the concept of
optimal weights in the council. Afterwards, we introduce and discuss
the MFM as well as previous results concerning the optimal weights for independent groups. Sections \ref{sec:weak}
and \ref{sec:Optimal-Weights-Low} contain the main results of this paper: we discuss the optimal weights
under the MFM for weak and strong coupling between voters, respectively.
Section \ref{sec:Independent-Clusters} treats several independent clusters of
groups, and Section \ref{sec:conclusion} concludes the paper. Finally, Section \ref{sec:Appendix} is an appendix which contains technical details regarding the democracy deficit, the optimal weights, and the MFM, as well as the proofs of the results presented in this paper.

\section{\label{sec:Def}Definition of Basic Concepts and Results}
In this section, we give a rigorous definition of voting measures, the democracy deficit, and the MFM, and discuss a few basic properties.
\subsection{The Setting}
Suppose the overall population is of size $N=N_{1}+\cdots+N_{M}$, whereas the group $\lambda $
has $N_{\lambda}$ voters, where the subindex $\lambda$ stands for
the group $\lambda\in\left\{ 1,\ldots,M\right\} $ . Let the two voting
alternatives be recorded as $\pm1$, $+1$ for `aye' and $-1$ for
`nay'. The vote of voter $i\in\left\{ 1,\ldots,N_{\lambda}\right\} $
in group $\lambda$ will be denoted by the variable $X_{\lambda i}$.
We will refer to the $N $-tuples $\left(x_{11},\ldots,x_{1N_{1}},\ldots,x_{M1},\ldots,x_{MN_{M}}\right)\in\left\{ -1,1\right\} ^{N}$
as voting configurations.

Throughout this article, we will study the asymptotic behaviour of the MFM, and we will assume that as the overall population
goes to infinity, so do the group populations, and that their relative sizes
compared to the overall population converge to fixed limits:
\begin{defn}
\label{def:group_sizes} We define the \emph{relative group size parameters}
for each group $\lambda$:
\[
\alpha_{\lambda}:=\lim_{N\rightarrow\infty}\frac{N_{\lambda}}{N}.
\]
\end{defn}

We will assume that $\alpha_{\lambda}>0$ holds for each group.

\begin{defn}\label{voting_margins}
For each group $\lambda$, we define the \emph{voting margin} $S_{\lambda}:=\sum_{i=1}^{N_{\lambda}}X_{\lambda i}$.
The overall voting margin is $S:=\sum_{\lambda=1}^{M}S_{\lambda}$.
\end{defn}
So there is a majority in group $\lambda $ in favour of a given proposal if $S_{\lambda}>0$.
Each group casts a vote in the council by applying the majority rule
to the group vote. Thus, the representative of group $\lambda $ votes `aye' if $S_{\lambda}>0 $. In other words,
\begin{defn}
The \emph{council vote of group} $\lambda$ is given by
\[
\chi_{\lambda}~:=~\chi(S_{\lambda})~=\begin{cases}
\,1, & \text{if }S_{\lambda}>0,\\
\,-1, & \text{otherwise.}
\end{cases}
\]
\end{defn}
Note that, for the MFM, the probability of a tie in each group goes to 0 as the population diverges to infinity. Hence, the decision to have group representatives cast a vote against the proposal in case of a tie as opposed to a vote in favour is inconsequential.

Each group $\lambda$
is assigned a voting weight $w_{\lambda}$. It is the goal of this paper to determine the `optimal' choice of these weights.

The weighted sum
\begin{align*} \sum_{\lambda=1}^{M}w_{\lambda}\chi_{\lambda} \end{align*}
is the \emph{council vote}.  Weights $w_{1},\ldots,w_{M} \in \IR$
together with a relative quota $q\in (0,1)$ constitute a weighted voting system
for the council, in which a coalition $A\subset\{1,2,\ldots,M\}$
is winning if
\begin{align*}
\sum_{\lambda \in A}\,w_{\lambda}~>~q\;\sum_{\lambda=1}^{M}\,w_{\lambda}.
\end{align*}

For the democracy deficit approach (see Definition \ref{def:democracy-deficit}), the relative quota in the council has no effect on the optimal weights\footnote{As the council vote only depends on the voting weights assigned to each group's representative and the vote cast by them, and the popular vote is of course independent of the relative quota, too, we see that the democracy deficit itself is invariant under all possible choices of the relative quota $q \in (0,1)$.}. We can take $q=1/2$, i.e.\! a simple
majority of the weighted votes suffices in the council. With $q=1/2$ the council vote is in favour of a proposal if
$\sum_{\lambda=1}^{M}w_{\lambda}\chi_{\lambda}>0$.

It is reasonable to choose the voting weights $w_{\lambda}$ in the
council in such a way, that the difference between the council vote
and a hypothetical referendum
\begin{align*}
\left|\,S-\sum_{\lambda=1}^{M}\,w_{\lambda}\,\chi_{\lambda}\,\right|
\end{align*}
is as small as possible in absolute value. We will call this magnitude the raw democracy deficit in order to distinguish it from the expectation we will be referring to as the `democracy deficit' later on.

There is clearly no choice of
weights which makes the raw democracy deficit \emph{uniformly} small over all possible voting
configurations. For any two choices of voting weights, there are some voting configurations where the first choice of weights has a lower raw democracy deficit and some voting configurations in which the other choice of weights is more favourable. Hence, all we can hope for is to make it small `on
average.' More precisely, we try to minimise the expected quadratic
deviation of $\sum_{\lambda=1}^{M}w_{\lambda}\chi_{\lambda}$ from
$S$.

To follow this approach, we have to clarify what we mean by `expected'
deviation, i.e.\! there has to be some notion of randomness underlying
the voting procedure.

We assume each individual has a set of deterministic and rational preferences concerning all possible issues which can be voted on.
However, the issue selected for a vote is assumed to be randomly chosen. If the choice is between two candidates for public office \emph{A} and \emph{B}, there is no fixed order in which the two must appear on the ballot; \emph{A} could correspond to the option $+1$ or $-1$.
Each yes/no question can be posed in different ways. Suppose the referendum is on a tax hike. The option $+1$ could correspond to implementing the hike, but it could also correspond to keeping the existing tax system. In short, there is no fundamental distinction between $+1$ and $-1$ beyond the fact that they represent two mutually exclusive choices. The voting configurations $(x_{11}, \ldots, x_{MN_M})$ thus provide information on the cohesion within the population. Is there a large majority in favour of one alternative or is the outcome close to a tie? The patterns in the voting configurations over all possible issues are described by a probability measure on the space $\{-1,1\}^N$.

These considerations lead to the following definition:
\begin{defn}
A \emph{voting measure} is a probability measure $\IP$ on the space
of voting configurations $\left\{ -1,1\right\} ^{N}=\prod_{\lambda=1}^{M}\,\{-1,1\}^{N_{\lambda}}$
with the symmetry property
\begin{align}
\mathbb{P}\left(X_{11}=x_{11},\ldots,X_{MN_{M}}=x_{MN_{M}}\right)~=~\mathbb{P}\left(X_{11}=-x_{11},\ldots,X_{MN_{M}}=-x_{MN_{M}}\right)\label{eq:symmetry}
\end{align}
for all voting configurations $\left(x_{11},\ldots,x_{MN_{M}}\right)\in\left\{ -1,1\right\} ^{N}$.
By $\mathbb{E}$ we will denote the expectation with respect to $\mathbb{P}$.
\end{defn}

The simplest voting measure is the $N$-fold product of the probability measures $P_0$ on $\{-1,1\}$ defined by
\begin{align*}
P_{0}(1) \coloneq P_{0}(-1) \coloneq \frac{1}{2},
\end{align*}
which models independence between all the voting results $X_{\lambda i}$, $\lambda=1,\ldots,M$, $i=1,\ldots,N_\lambda$.
In this much analysed case, known as the \emph{impartial culture}, we have
\begin{align*}
\mathbb{P}\left(X_{11}=x_{11},\ldots,X_{MN_{M}}=x_{MN_{M}}\right)~=~\prod_{\lambda=1}^{M}\prod_{i=1}^{N_{\lambda}}\,P_{0}(X_{\lambda i}=x_{\lambda i})~=~\frac{1}{2^{N}}\,
\end{align*}
for all voting configurations $\left(x_{11},\ldots,x_{MN_{M}}\right)\in\left\{ -1,1\right\} ^{N}$, i.e.\! each voting configuration occurs with the same probability.

Once a voting measure is given, the quantities $X_{\lambda i}$, $S_{\lambda}$,
$\chi_{\lambda}$,  the raw democracy deficit, etc.\!, are random variables defined on the same
probability space $\left\{ -1,1\right\} ^{N}$.

An extension of these binary voting systems and the probabilistic voting models describing the voters' behaviour is taking into account the possibility of abstainig from a vote. This can happen at both the population level, where each voter can decide if they want to vote in favour, against, or abstain from voting, as well as at the council level, where each representative can abstain from voting if their group is tied about the issue at hand. The latter does not present a meaningful distinction compared to the setup without abstentions considered in this article, as under the MFM (as well as other voting models such as the collective bias model considered in \cite{KT2021_CBM}), the probability of a draw in a given group goes to 0 as the population goes to infinity. Therefore, abstentions will not occur in the large population limit. Allowing for abstentions at the population level presents a number of challenges, such as the question of how to define a voting model which is unclear even in the simplest case of independent voting (see \cite{FelsMach1997,Birkmeie2011,BirKaePu2011}). It is an interesting question to consider in future research.

\subsection{Democracy Deficit and Optimal Weights}

With the concept of a voting measure at our disposal, we can formally
define the democracy deficit. For more details on the topic of democracy deficit and optimal weights, see \cite{KT2021_CBM}.
\begin{defn}
\label{def:democracy-deficit}The \emph{democracy deficit} given a
voting measure $\mathbb{P}$ and a set of weights $w_{1},\ldots,w_{M} \in \IR$
is defined by
\begin{align*}
\Delta_{1}~=\Delta_{1}(w_{1},\ldots,w_{M})~:=~\mathbb{E}\left[\left(S-\sum_{\lambda=1}^{M}w_{\lambda}\chi_{\lambda}\right)^{2}\right]\,.
\end{align*}
We call $(w_{1},\ldots,w_{M})$ \emph{optimal weights} if they minimise
the democracy deficit, i.e.\!
\begin{align*}
\Delta_{1}(w_{1},\ldots,w_{M})~=~\min_{(v_{1},\ldots,v_{M})\in\mathbb{R}^{M}}\;\Delta_{1}(v_{1},\ldots,v_{M}).
\end{align*}
\end{defn}

\begin{rem}
For any weighted voting system, we obtain an equivalent voting system by multiplying each voting weight by the same positive constant and leaving the relative quota unchanged. Therefore, whenever we speak of the uniqueness of the vector of optimal weights,
it shall be understood to mean `uniqueness up to multiplication by
a positive constant.'
\end{rem}

It is mathematically convenient to allow \emph{real} numbers as weights. In practice, however, \emph{integer valued} weights are more convenient, if not required. Fortunately this can always be implemented as the following Lemma shows.

\begin{lem}
   Given a weighted voting system with weights $w_{1}, w_{2},\ldots,w_{K}\in\IR $, there is always an equivalent voting system with weights
   $\tilde{w}_{1}, \tilde{w}_{2},\ldots,\tilde{w}_{K}\in\IN $
\end{lem}
\begin{proof}
    Since the space of possible voting configurations $\left\{ -1,1\right\} ^{N}$ is finite, the voting system is unchanged by very small changes in the weights. Thus we may suppose without loss of generality that the weights are rational numbers. By multiplying the weights as well as the quota by an integer, the smallest common denominator, we obtain an equivalent voting system with integer weights. 
\end{proof}

Note that the democracy deficit depends both on the weights and the voting measure. We observe that minimising the democracy deficit implies
that the magnitude and sign of the council vote approximate well
the magnitude and sign of the popular vote. We do not merely
wish to achieve agreement between the two outcomes in the binary sense
but a rather stronger property: the population should observe that the
council follows the public opinion as closely
as possible. This stands in contrast to the criterion of minimising the probability that the binary council decision differs from the decision made by a referendum, which is less strict in the sense that for a favourable public opinion of 51\%, a 51\% vote in the council and a 100\% vote would be considered equally satisfactory. However, a 100\% vote in the council would not be a good representation of public opinion at all. The 49\% minority might feel they are not represented in the council at all, giving rise to populist anti-elite sentiment among them. Viewed from this perspective, adjusting the voting outcomes in the council in such a way that they follow the popular opinion as closely as possible is a worthwhile goal.

Our objective is to choose the weights such that the democracy deficit is minimised.
By taking partial derivatives of $\Delta_{1}$ with respect to
each $w_{\lambda}$ and equating each one to 0, we obtain a system of linear equations that characterizes
the optimal weights. Indeed, for $\lambda=1,\ldots,M$,
\begin{equation}
\sum_{\nu=1}^{M}\;\mathbb{E}\left(\chi_{\lambda}\chi_{\nu}\right)w_{\nu}~=~\mathbb{E}\left(\chi_{\lambda}S\right)\,.\label{eq:LES}
\end{equation}
Defining the matrix $A^N$, the weight vector $w$ and the vector $b^N$
on the right hand side of \eqref{eq:LES} by
\begin{align}
A^N~ & :=~\left(A^N_{\lambda\nu}\right)_{\lambda,\nu=1,\ldots,M}~:=~\left(\mathbb{E}\left(\chi_{\lambda}\chi_{\nu}\right)\right)_{\lambda,\nu=1,\ldots,M},\label{eq:A0}\\
w^N~ & :=~\left(w^N_{\lambda}\right){}_{\lambda=1,\ldots,M},\nonumber \\
b^N~ & :=~\left(b^N_{\lambda}\right)_{\lambda=1,\ldots,M}~:=~\,\left(\mathbb{E}\left(\chi_{\lambda}S\right)\right)_{\lambda=1,\ldots,M}\,,\nonumber
\end{align}
we may write \eqref{eq:LES} in matrix form as
\begin{align}
A^N\;w^N~=~b^N\,.\label{eq:LESM}
\end{align}
A solution $w$ of \eqref{eq:LESM} is a minimum of $\Delta_1$  if the matrix $A^N$,
the Hessian of $\Delta_{1}$, is positive definite. In this case,
the matrix $A^N$ is invertible, and consequently there is a unique tuple
of optimal weights, namely the unique solution of \eqref{eq:LESM}. However, due to the difficulties associated with the calculation of the above quantities for finite (but large) populations, we will calculate the asymptotic weights in the limit that the group populations all go to infinity in accordance with Definition \ref{def:group_sizes}. See Section \ref{subsec:more_dem_def} of the Appendix for the technical details concerning the asymptotic behaviour of the democracy deficit and the optimal weights.

\subsection{\label{subsec:Model}Mean-Field Model}

In statistical mechanics, the MFM\footnote{This model is also called the `Curie-Weiss model', named after the
physicists Pierre Curie and Pierre Weiss.} is usually defined for a single set of spins or binary random variables.
There is an energy function, also called Hamiltonian, that assigns to
each spin configuration $x=\left(x_{1},\ldots,x_{N}\right)\in\left\{ -1,1\right\} ^{N}$
a real number
\begin{equation}
\mathbb{H}(x):=-\frac{J}{2}\left(\frac{1}{\sqrt{N}}\sum_{i=1}^{N}x_{i}\right)^{2}.\label{eq:Ham_single_group}
\end{equation}
This energy function determines the `cost' of the configuration.
Less costly configurations are thought of as more common. The only
parameter of the model is  the parameter $J\geq 0$ which reflects the strength of the coupling between spins. In physics, $J$ can be interpreted as an `inverse temperature' parameter.

The probability measure on the space of configurations $\left\{ -1,1\right\} ^{N}$
is a so called Gibbs measure that assigns each configuration $x$ the probability
\begin{align}\label{eq:sCW}
\mathbb{P}\left(x\right)~=~\mathbb{P}_{J}\left(x\right)~:=~Z^{-1}\exp\left(-\mathbb{H}(x)\right).
\end{align}
$Z$ is a normalisation constant which makes $\IP$ a probability measure. $Z$
 depends on both $J$ and the number of spins $N$.
The minus sign in \eqref{eq:sCW} makes configurations $x \in \{-1,1\}^N$with lower energy levels $\mathbb{H}(x)$
more probable under the measure $\mathbb{P}$. This means configurations with a large majority of $+1$'s or a large majority of $-1$'s are more likely, i.e.\! there is a tendency to align with other voters. This tendency is stronger for large $J$.

In \cite{SRL}, 
this single-group
model was employed to study two-tier voting systems. The limitation of such an
approach is that each group is described by a separate single-group
model, thus precluding the possibility of studying correlated voting
across group boundaries.

In order to study the coupling between voters belonging to different
groups, we need to define a model with several different
sets of spins that potentially interact with each other in different
ways. Instead of a single inverse temperature parameter, there is
a coupling matrix that describes the interactions between voters. We will call this
matrix
\[
\mathbf{J}:=(J_{\lambda\nu})_{\lambda,\nu=1,\ldots,M}.
\]
$J_{\lambda\lambda}$ describes the coupling of voters inside the group $\lambda $, and $J_{\lambda\nu}, \lambda\not=\nu$, stands for the coupling of voters from group $\lambda$ and those from group $\nu$.

Just as in the single-group model, there is a Hamiltonian function
that assigns each voting configuration a certain energy level. This
energy level can be interpreted as the cost of a given voting configuration
in terms of the conflict between different voters. Voters tend to
vote in such a way that the conflict is minimised. For each voting
configuration $\left(x_{11},\ldots,x_{MN_{M}}\right)\in\left\{ -1,1\right\} ^{N}$,
we define
\begin{align}
\mathbb{H}\left(x_{11},\ldots,x_{MN_{M}}\right) & :=-\frac{1}{2}\sum_{\lambda,\nu=1}^{M}J_{\lambda\nu}\;
\left(\frac{1}{\sqrt{N_{\lambda}}}\sum_{i=1}^{N_{\lambda}}x_{\lambda i}\right)\,\left(\frac{1}{\sqrt{N_{\nu}}}\sum_{j=1}^{N_{\nu}}x_{\nu j}\right)\,.
\label{eq:Hamilton}
\end{align}

In the single-group model, we had to assume that $J\geq 0$ in order to get a decent probability measure. In the multi-group case we need analogously $\bJ\geq 0$, in the sense that the symmetric matrix $\bJ$ is positive semi-definite\footnote{Note that the assumption of a symmetric coupling matrix by itself represents no constraint of the model since for a non-symmetric coupling matrix $\bJ'$ there is an equivalent mean-field model with a coupling matrix $\bJ$, where the off-diagonal entries are
\[
J_{\lambda\nu}=\frac{J_{\lambda\nu}'+J_{\nu\lambda}'}{2},\quad\lambda\neq\nu.
\]
}, i.e.\! $\langle x,\bJ x \rangle\geq 0$ holds for all vectors $x \in \IR^M$. The symbol $\langle \, \cdot\,,\,\cdot\, \rangle$ stands for the Euclidean inner product on $\IR^M$. Note that this implies $J_{\nu\nu}\geq 0$ for all $\nu$, but the off-diagonal entries $J_{\nu\lambda}, \nu\not=\lambda$, can be positive or negative.

Instead of each voter interacting with each other voter in the exact
same way, voters in different groups $\lambda,\nu$ are coupled by
a coupling constant $J_{\lambda\nu}$. These coupling constants subsume
the `inverse temperature' parameter $J$ found in the single-group
model. We note that depending on the signs of the coupling constants
$J_{\lambda\nu}$ different voting configurations have different energy
levels assigned to them by $\mathbb{H}$. If all coupling constants
are positive, there are two voting configurations that have the lowest
energy levels possible: $(-1,\ldots,-1)$ and $(1,\ldots,1)$. All
other voting configurations receive higher energy levels. The highest
levels are those where voters are evenly split (or closest to it in
case of odd group sizes). This represents the assumed tendency of
voters to cooperate with each other if they are positively coupled.
\begin{defn}\label{multi-group_MFM}
Let $\bJ$ be a positive semi-definite $M\times M$ matrix, and let $\mathbb{H}$ be defined by \eqref{eq:Hamilton}. The mean-field
probability measure $\mathbb{P}$, which gives the probability of each
of the $2^{N}$ voting configurations, is defined by
\begin{align}
\mathbb{P}_{\bJ}\left(X_{11}=x_{11},\ldots,X_{MN_{M}}=x_{MN_{M}}\right):=Z^{-1}\exp\left(-\mathbb{H}\left(x_{11},\ldots,x_{MN_{M}}\right)\right)\label{eq:Pbeta}
\end{align}
for each $\left(x_{1},\ldots,x_{N}\right)\in\left\{ -1,1\right\} ^{N}$. $Z$ is a normalisation constant
which depends on $N$ and $\bJ$. $\bJ$ is called the \emph{coupling matrix} of the model. Whenever the matrix $\bJ $ is clear from the context, we drop the subscript and write $\IP$ instead of $\IP_{\bJ} $. The expectation with respect to $\IP_{\bJ} $ is called $\IE_{\bJ} $ or simply $\IE $.
\end{defn}
The mean-field measure is indeed a voting measure, as can be seen
from the definition of the Hamiltonian \eqref{eq:Hamilton}. We note that impartial culture is a special case of the MFM if we set $\bJ = 0$. The single-group MFM is another special case of the multi-group version (for the number of groups $M=1$, Definition \ref{multi-group_MFM} reduces to the probability measure given in \eqref{eq:sCW}).

In the field of statistical physics, the regimes of the MFM are called `temperature regimes'
because the single-group model
has only a single parameter $J\geq 0$ which can be interpreted
as the inverse temperature of the spin system. In the present context,
different temperatures correspond to different intensities of coupling
between voters. The suitably normalised group voting margins (see Definition \ref{voting_margins}) behave differently in each of the three regimes, which constitutes an emergent phenomenon rather than being an explicit part of Definition \ref{multi-group_MFM} of the model. A high temperature means there is a lot of disorder
or confusion, and the voters mostly make up their own minds.
There may still be \emph{some} tendency to vote alike; however, the typical majorities are not large. We will
call this the `weak coupling regime,' which is characterised by the matrix $\bI-\bJ$ being positive definite (with $\bI$ being the identity matrix), i.e.\! $\bI-\bJ > 0$.
At low temperatures, voters
want to align with others. As a result, votes will be strongly correlated, with large majorities in favour of one alternative being typical.
We will call this the `strong coupling regime,' for which $\bI-\bJ$ is not positive semi-definite, i.e.\! $\bI-\bJ \ngeq 0$. See Section \ref{subsec:Regimes} of the Appendix for a more thorough discussion of the model and its regimes.

At this point we will very briefly discuss the behaviour of the single-group model, which due to its simple nature is more easily understood in intuitive terms. The weak coupling regime is defined by $J\in\left[0,1\right)$. This regime, just as in the multi-group model which is the topic of this article, is characterised by small majorities which manifest through expected voting margins $\IE\left|S\right|$ that behave asymptotically like $C_{J}\sqrt{N}$, which is of smaller order than the population $N$. However, the prefactor $C_{J}$ is independent of $N$ but depends on $J$, and in fact $\lim_{J\nearrow1}C_{J}=\infty$ holds. So while it is true that for fixed $J\in\left[0,1\right)$ the majorities are typically small, it should also be noted that they become larger as $J\nearrow1$. This observation is complemented by the behaviour of the strong coupling regime $J\in\left(1,\infty\right)$: while the typical majority is of order $N$, i.e.\! large, by choosing $J$ close enough to 1, we can reduce the macroscopic majority and come arbitrarily close to a tie. $\IE\left|S\right|$ behaves asymptotically like $C_{J}N$ for all $J>1$ and $\lim_{J\searrow1}C_{J}=0$ is satisfied. The MFM thus covers the entire range of typical majorities being very small to very large, nearly unanimous. In our opinion, this flexibility makes it an interesting model to study and to apply to the problem of optimal weights in two-tier voting systems.

Before we state the results concerning the optimal weights under the MFM, we recapitulate the corresponding results for independent groups.

\subsection{Independent Groups} \label{Independent_Groups}

The case of independent groups was analysed in \cite{SRL}.
In that article, each group was described by a separate MFM. Independent
groups can also be described by the multi-group MFM by choosing a diagonal coupling matrix $\bJ$. In this case,
the coefficient matrix $\eqref{eq:A0}$ in the linear equation system
that characterises the optimal weights is diagonal and the entries
are all equal to 1. Hence, the solution is simple: for each group
$\lambda$, the optimal weight $w_{\lambda}$ is given by
\[
w_{\lambda}=\IE\left(\chi_{\lambda}S\right)=\IE\left(\chi_{\lambda}S_{\lambda}\right)=\IE\left|S_{\lambda}\right|.
\]
The last expression above behaves differently depending on the regime
of the model. In the single-group model, the weak coupling regime corresponds to $J<1$, where $J$ is the coupling constant in \eqref{eq:Ham_single_group}, and the strong coupling regime is $J>1$. In the weak coupling regime, $\IE\left|S_{\lambda}\right|$
behaves like  $C_{\lambda}\sqrt{N_{\lambda}}$  for large populations (we shall say `is asymptotically equal to;' see Definition \ref{def:appr} in Section \ref{subsec:Regimes} of the Appendix).
Hence, the optimal weight for each group is proportional to the square
root of each group's population, possibly with different constants
$C_{\lambda}$ for each group. This is qualitatively similar to the
prescription made by Penrose's square root law. This result can be
interpreted as the weak coupling regime being close enough to independence
so as not to affect the optimal weights, at least in a qualitative
sense. Note that, crucially, this is only true as long as the groups
are independent. If we relax this assumption, the square root law
fails to hold as we will see in Section \ref{sec:weak}.

In the strong coupling regime, $\IE\left|S_{\lambda}\right|$ is asymptotically
equal to $C_{\lambda}N_{\lambda}$. Thus, the optimal weight
is proportional to the group's population. This result also fails
to hold in the more general setting considered in the present article,
which features dependence between voters belonging to different groups. We will see in Section \ref{sec:Optimal-Weights-Low}
that introducing dependence in the strong coupling regime can lead
to the optimal weights not being uniquely determined.

Since under independent groups each group can be in a different regime,
we see that having a structure of strong coupling between members
of a group is favourable with regard to the optimal weight. We will
return to the question of different groups being in different regimes
in Section \ref{sec:Independent-Clusters}, where we generalise the results presented here to independent
clusters of several groups each.

The main qualitative aspect we would like to note concerning the optimal
weights under independent groups is that the they are proportional
to either the population or the square root of the population. As we will see in the next two sections, this is no longer
the case when the groups are not independent. We will see that, in
some cases, a constant summand appears in the formula for the optimal
weights. In other cases, the optimal weights are no longer uniquely
determined. Under some circumstances, the optimal weights turn out negative. The new features introduced by dependent groups are manifold.

\section{Optimal Weights for the Weak Coupling Regime}\label{sec:weak}

In this section, we will analyse the optimal weights for the MFM in its weak coupling regime. The coupling between voters in this regime is -- as implied by the name of the regime -- fairly weak. Although the voters are not independent, they do tend to make up their mind on their own for the most part. Another way to describe this regime is to say there is a large amount of turmoil in the overall population, with polarised opinions.

We will analyse three scenarios, each one characterised by the form of the coupling matrix. But first, we will discuss what it means for the MFM to be in the weak coupling regime.

\subsection{Basics}\label{subsec:Basics}

To describe a given system of groups of voters we have to adjust the quantities $J_{\lambda\nu}$ according to
the concrete situation. However, typically voting weights will be described in a kind of constitution or founding treaty, which should be designed for a long time period. As time goes by, interactions between the founding members may change, new groups may join the union, and groups may leave the union. So, for defining voting weights in a founding document, it seems appropriate to consider a `typical', simplified set of coupling parameters $J_{\lambda\nu}$. In fact, one of the scenarios we will explore will feature internal coupling $J_{0}:=J_{\lambda\lambda}$ and coupling between groups $\bar{J}:=J_{\lambda\nu} $ independent of the groups $\lambda \neq \nu$ involved.  More precisely, we consider the following model:

\begin{defn}\label{def:hom}
   An MFM with coupling matrix $\mathbf{J}(J_{0},\bar{J}) $ given by
\begin{align*}
  \bJ(J_{0},\bar{J})_{\nu\lambda}~:=~\left\{
                      \begin{array}{ll}
                        J_{0}, & \hbox{for $\nu=\lambda $,} \\[2mm]
                        \bar{J}, & \hbox{otherwise,}
                      \end{array}
                    \right.
\end{align*}
with $J_{0}>0 $ and $\bar{J}\in\IR$ is called a \emph{balanced model}. It is called \emph{homogeneous} if $J_{0}=\bar{J}$.
\end{defn}
\begin{lem}\label{lem:hom}
   For the balanced model, the matrix $\bJ=\bJ(J_{0},\bar{J})$ is positive definite if and only if
\begin{align*}
   -\frac{J_{0}}{M-1}~<~\bar{J}~<~J_{0}.
\end{align*}
For the homogeneous model (with $J_{0}=\bar{J}>0$), the matrix $\bJ$ is always positive semi-definite but not positive definite.
\end{lem}
Lemma \ref{lem:hom} follows directly from Lemma \ref{lem:balanced_matrix} (see Section \ref{subsec:Matrix} of the Appendix). This lemma gives bounds on the values of the two constants $J_0$ and $\bar{J}$ within which the coupling matrix is positive semi-definite, an assumption we made when we defined the MFM. As we see, the bounds are not symmetric with respect to the origin: the coupling constant $\bar{J}$ which defines coupling between groups can be up to $J_0$, the coupling between voters in the same group. As a lower bound, we have a constant smaller in absolute value which depends on the number of groups in the model. We can interpret this lemma as stating that the coupling between groups can be no stronger than the coupling within groups.

Before we turn to the optimal weights, we identify the two regimes for balanced models.
\begin{lem}\label{lem:balancedregime}
For the balanced  model $\bJ(J_{0},\bar{J})$ (with $-J_0/(M-1)<\bar{J}\leq J_{0}$), we have:
\begin{enumerate}
\item If $\bar{J}\geq 0$, then
   the coupling matrix $\bJ(J_{0},\bar{J})$ is in the weak coupling regime if
\begin{align*}
   J_{0}+(M-1)\bar{J}~<~1,
\end{align*}
and in the strong coupling regime if
\begin{align*}
   J_{0}+(M-1)\bar{J}~>~1.
\end{align*}
\item If $\bar{J}< 0$, then
   the coupling matrix $\bJ(J_{0},\bar{J})$ is in the weak coupling regime if
\begin{align*}
   J_{0}+|\bar{J}|~<~1,
\end{align*}
and in the strong coupling regime if
\begin{align*}
   J_{0}+|\bar{J}|~>~1.
\end{align*}

\end{enumerate}
\end{lem}
Lemma \ref{lem:balancedregime} follows from Lemma \ref{lem:balanced_matrix}. The lemma says that the sign of the inter-group coupling constant $\bar{J}$ affects the range of $|\bar{J}|$ that stays within the weak coupling regime: for non-negative inter-group couplings, the value has to be fairly small to stay in the regime, whereas for negative $\bar{J}$, there is more leeway. In a sense, this condition is complementary to the condition in Lemma \ref{lem:hom} which characterises the positive definiteness of $\bJ(J_{0},\bar{J})$.

\subsection{\label{subsec:Positive-Coupling-High} Friendly World}
In this first scenario, we consider an MFM with balanced coupling matrix $\bJ(J_{0},\bar{J}) $ with positive coupling between all groups, i.e.\! with $0 \leq \bar{J} \leq  J_{0}$. So the matrix $\bJ(J_{0},\bar{J}) $ is positive semi-definite and there are only \emph{positive} correlations between votes both within the same group and across group boundaries. This scenario models a union of groups that relate in a friendly way to each other.

We assume that $\bJ(J_{0},\bar{J}) $ is in the weak coupling regime, so $J_{0}+(M-1)\bar{J}<1 $ by Lemma \ref{lem:balancedregime}. For this model, we prove an extension of Penrose's square root law.

We set
\begin{align*}
  \rho~&:=~\lim_{N\to\infty}\IE(\chi_{1}\chi_{2})~=~ \frac{2}{\pi}\,\arcsin\left(\frac{\bar{J}}{1-J_{0}-(M-2)\bar{J}}\right),\\
\tau~&:=~\frac{\bar{J}}{1-J_{0}-(M-2)\bar{J}},\\
\eta&:=\sum_{\lambda=1}^{M}\sqrt{\alpha_{\lambda}} \,.
\end{align*}
The value for $\rho$ given above is proved in Proposition \ref{prop:arcsin}.
\begin{rem}
   For $\jj\geq 0 $ in the weak coupling regime, we have $0\leq \tau < 1 $, so the expression $\arcsin(\tau) $ above is well defined. The correlation $\rho $ can assume any value in $[0,1)$. If the system is `close to strong coupling', in the sense that $J_{0}+(M-1)\bar{J} \nearrow 1 $, the correlation $\rho$ approaches $1 $.
\end{rem}

\begin{thm}
\label{thm:identical_pos_corr}Suppose the coupling matrix $\bJ(J_{0},\bar{J})$ is in the weak coupling regime and $0 \leq \bar{J} \leq J_0$. Then the
optimal weights are given by
\begin{align}\label{eq:friendly_world_opt_weight}
w_{\lambda}=D_{1}\sqrt{\alpha_{\lambda}}+D_{2}\eta
\end{align}
for each group $\lambda=1,\ldots,M$, where
\begin{align*}
D_{1} & =\big(1+\left(M-1\right)\rho\big)\left(1-J_{0}-\left(M-1\right)\bar{J}\right),\\
D_{2} & =\left(1+\left(M-2\right)\rho\right)\bar{J}-\rho\left(1-J_{0}\right).
\end{align*}
The coefficient $D_{1}$ is positive, and $D_{2}\geq0$ with equality
if and only if $\bar{J}=0$.
\end{thm}

\begin{proof}
The theorem is proved in Section \ref{subsec:Proof-thm-identical_pos_corr}.
\end{proof}

Note that the coefficients $D_1$ and $D_2$ only depend on the coupling matrix but not the relative sizes of the groups.

Theorem \ref{thm:identical_pos_corr} can be regarded as a generalisation of the square root law by Penrose to the case of weak dependence between groups.  The theorem states that the optimal weights are composed of a summand
proportional to the square root of the group's population and a constant
summand equal for all groups. The constant summand is 0 if and only
if the groups are independent. Thus, we recover the square root law
from \cite{SRL} for $\bar{J}=0$. For dependent voters across group
boundaries, there is no pure square root law. Instead, the optimal
weight is given by a term equal for each group and a term proportional
to the square root of the group's population. It is important to note that the dependence between voters in different groups is indeed the sole source of the constant term $D_2 \eta$ in the formula for the optimal weights.

We also contrast Theorem \ref{thm:identical_pos_corr} with the optimal weight under the collective bias model given in Theorem 21 in \cite{KT2021_CBM}. Said theorem gives the optimal weights in a similar setting as the friendly world under weak coupling for the MFM, where there is positive correlation between votes in different groups, but said correlation is not very strong. The optimal weights for each group $\lambda$ are of the form
\begin{align}\label{opt_CBM_non-tight}
w_{\lambda}=C_{1} \alpha_{\lambda}+C_{2}
\end{align}
with constants $C_1$ and $C_2$ that depend on the voting measure defining the collective bias model and the number of groups, but not on the size $\alpha_\lambda$ of group $\lambda$. The two prescriptions for optimal weights have in common that there is a constant term $D_2 \eta$ or $C_2$ which is equal for all groups. The presence of this constant term is owed entirely to the dependence between votes belonging to different groups (cf.\! the results for independent groups presented in Section \ref{Independent_Groups}), and it is a general feature of optimal weights for any voting measure not only the MFM and the collective bias model. The two prescriptions differ in the other summand: the collective bias model leads to optimal weights with one summand being proportional to the size of the group, whereas for the MFM the summand which depends on the group's size is proportional to the square root of the group's size. This feature distinguishes the two models from the point of view of the optimal weights in a two-tier voting systems. There is no regime of the MFM that produces a formula for optimal weights of the form \eqref{opt_CBM_non-tight} and no version of the collective bias model that produces a formula of the form \eqref{eq:friendly_world_opt_weight}.

\subsection{Hostile World}\label{subsec:hostile}

Now we consider a scenario  where all groups are antagonistic towards
each other. More precisely, we investigate an MFM with coupling matrix $\bJ(J_{0},\bar{J})$ such that $J_{0}>0$ but $\bar{J}\leq 0$. Note that the voters within each group are still positively correlated, as this is a general feature of the MFM.

Again, we suppose that the system is in the weak coupling regime; in other words, by Lemma \ref{lem:balancedregime},
\begin{align}\label{eq:host}
1-J_{0}+\bar{J}> 0 \,.
\end{align}

\begin{prop}\label{prop:rho}
   In the model $\bJ(J_{0},\bar{J})$ with $\bar{J}\leq 0$, we have in the weak coupling regime
\begin{align*}
   -\frac{1}{M-1}~<~\rho~=\lim_{N\to\infty}\IE(\chi_{1}\chi_{2})~\leq~0.
\end{align*}
\end{prop}
It may be surprising at first glance that $\rho $ is bounded from below away from the value $-1$, and in particular that the lower bound goes to $0$ if the number of groups $M$  goes to infinity. In a sense, the reason behind this phenomenon is the ancient
wisdom `the enemy of my enemy is my friend.' For example, if there are three groups, two of them must necessarily agree in a specific vote.

Proposition \ref{prop:rho} follows by `abstract' results on general exchangeable sequences (see e.g.\! \cite{Aldous}).
In Section \ref{subsec:rho}, we give a `concrete' proof of Proposition \ref{prop:rho}.

\begin{thm}
\label{thm:hostile_world}
For the model $\bJ(J_{0},\bar{J})$ with $\bar{J} \leq 0 < J_0$ in the weak coupling regime, the optimal weights are

\begin{equation}
w_{\lambda}~=~D_{1}\sqrt{\alpha_{\lambda}}\,+\,D_{2}\eta\label{eq:hostile_world_opt_weight}
\end{equation}
for each group $\lambda=1,\ldots,M$, where again
\begin{align*}
D_{1} & =\left(1+\left(M-1\right)\rho\right)\left(1-J_{0}+\left(M-1\right)\bar{J}\right),\\
D_{2} & =\left(1+\left(M-2\right)\rho\right)\bar{J}-\rho\left(1-J_{0}\right).
\end{align*}
Above the coefficient $D_{1}$ is positive, and $D_{2}\leq0$ with equality
if and only if $\bar{J}=0$.
\end{thm}
The proof is the same as for the `friendly world' scenario, see Section \ref{subsec:Proof-thm-identical_pos_corr}.

The expressions for the optimal weights in \eqref{eq:friendly_world_opt_weight} and in \eqref{eq:hostile_world_opt_weight} are the same. However, the sign of $\rho $ is positive in the first case and negative in the latter.

The optimal weights given by formula \eqref{eq:hostile_world_opt_weight}
are the sum of a term proportional to the square root of the group's
population and a \emph{negative} offset equal for all groups. This offset
is the product of a factor $D_{2}$ which depends on the coupling
matrix $J$ and a factor $\eta$ which depends on the distribution
of the groups' sizes. Very small groups may receive a negative weight,
whereas the largest groups always receive a positive weight.

Negative weights make sense in statistical estimation problems and for automated preference aggregation where the voting weights are not made public. For political practice, they are completely inappropriate. If a voter had a negative weight, they would choose to misrepresent their true preferences! In this case, it is impossible to achieve the theoretical minimum of the democracy deficit given by
\eqref{eq:LESM}.

A possible solution for the case of negative voting weights which result from solving the problem of optimal weights which minimise the democracy deficit is to consider a different criterion for `fair' voting weights. An example of such a criterion is the minimisation of the probability that the council makes a contrary decision to the popular vote. This is an avenue for future research.

As in the previous scenario, letting the groups be independent by setting
$\bar{J}=0$ reduces \eqref{eq:hostile_world_opt_weight} to the square
root law.

Negative optimal weights arise only under dependence of votes belonging to different groups. The article \cite{KT2021_CBM} contains an extensive analysis (see Section 9 of \cite{KT2021_CBM}) of the circumstances under which negative weights arise in the collective bias model.

\subsection{\label{subsec:Two-Clusters-high} Split World}

In this scenario, the world is split into two blocks or clusters. Let the clusters contain $M_{i},\, i=1,2$,
groups so that $M_{1}+M_{2}=M$. Without loss of generality, assume
that the cluster $C_{1}$ contains the first $M_{1}$ groups and $C_{2}$,
the last $M_{2}$. Let the coupling matrix have the block matrix form
\begin{equation}
\bJ=\left(\begin{array}{cc}
J^{1} & B\\
B^{T} & J^{2}
\end{array}\right),\label{eq:two_cluster_coupling}
\end{equation}
where $J^{i}\in\IR^{M_{i}\times M_{i}},i=1,2,$ are matrices of the form $\bJ(J_{0},\jj) $ of dimension $M_{i}$ with $\jj>0 $ and let
$B=-\bar{J}\,\boldsymbol{1}_{M_{1}\times M_{2}}$. We use the notation
$\boldsymbol{1}_{m\times n}$ to denote an $m\times n$ matrix with
all entries equal to 1.

Hence, voters belonging to the same group have a coupling of $J_{0}$,
voters in different groups of the same cluster are coupled positively
with strength $\bar{J}$, and voters belonging to groups in different clusters are coupled negatively with strength $-\bar{J}$.  According to Lemma \ref{lem:Matrix},
the matrix $\bJ $ is positive definite if $\jj < J_{0}$ and belongs to the weak coupling regime if $J_{0}+(M-1)\jj<1 $.

Let $\rho$ stand for the intra-cluster correlation $\lim_{N\to\infty}\IE\left(\chi_{1}\chi_{2}\right)$ (assuming $M_{1}\geq 2$)
and let $\bar{\eta}:=\sum_{\lambda\in C_{1}}\sqrt{\alpha_{\lambda}}-\sum_{\lambda\in C_{2}}\sqrt{\alpha_{\lambda}}$.

The optimal weights are as follows:
\begin{thm}
\label{thm:two_cluster_weights}For a coupling matrix $\bJ $ as in \eqref{eq:two_cluster_coupling}
with $0\leq \jj < J_{0}$ and $J_{0}+(M-1)\jj<1 $, the
optimal weights are given by
\begin{equation}
w_{\lambda}=D_{1}\sqrt{\alpha_{\lambda}}+\left\{
                                           \begin{array}{ll}
                                             D_{2}\;\bar{\eta}, & \hbox{for $\lambda\in C_{1} $,} \\
                                            - D_{2} \;\bar{\eta}, & \hbox{for $\lambda\in C_{2} $,}
                                           \end{array}
                                         \right.
\label{eq:two_cluster_weights}
\end{equation}
for each group $\lambda=1,\ldots,M$, where
\begin{align*}
D_{1} & =\left(1+\left(M-1\right)\rho\right)\left(1-J_{0}-\left(M-1\right)\bar{J}\right),\\
D_{2} & =\left(1+\left(M-2\right)\rho\right)\bar{J}-\rho\left(1-J_{0}\right).
\end{align*}
The coefficient $D_{1}$ is positive, and $D_{2}\geq0$ with equality
if and only if $\bar{J}=0$.
\end{thm}

\begin{proof}
The theorem is proved in Section \ref{subsec:Proof-thm_two_cluster}.
\end{proof}
The optimal weights have identical coefficients $D_{1}$ and $D_{2}$
to the scenarios discussed before.  However, instead of $\eta$, the sum of all $\sqrt{\alpha_{\lambda}}$,
we have $\bar{\eta}$, the difference between the sums of the $\sqrt{\alpha_{\lambda}}$
belonging to each cluster. As a rule, either $\bar{\eta}$ or
$-\bar{\eta}$ will be negative. Therefore, there are cases where
one or more groups are assigned a negative voting weight. This happens
when the term $\pm D_{2}\bar{\eta}$
is negative and larger in absolute value than $D_{1}\sqrt{\alpha_{\lambda}}$.

As we discussed in Section \ref{subsec:hostile}, negative weights are not acceptable for real life political systems.
If there are no groups small enough for a negative weight, then the
democracy deficit can be minimised as in the friendly world scenario. We
will say that $C_{1}$ is `larger' and `more uniformly sized'
than $C_{2}$ if $\bar{\eta}>0$, even though strictly speaking $\bar{\eta}>0$
can hold even if cluster 1 represents less than half the overall population.
By the formula \eqref{eq:two_cluster_weights}, groups belonging to
the larger of the two clusters receive a weight composed of the sum
of a term proportional to their population's square root, $D_{1}\sqrt{\alpha_{\lambda}}$,
and a constant term, $D_{2}\left|\bar{\eta}\right|$, equal for each
group in that cluster. The groups belonging to the smaller cluster
also receive a weight given by such a sum; however, the constant term
is $-D_{2}\left|\bar{\eta}\right|$. As in the previous scenarios, if the
groups are independent, then $D_{2}=0$, and we recover the square
root law. In addition to that, if $\bar{\eta}=0$, then the weights
are proportional to the square roots, even if the groups are not independent.
$\bar{\eta}=0$ can occur even if there are different numbers of groups
in each cluster and they represent different proportions of the overall
population.

\section{\label{sec:Optimal-Weights-Low}Optimal Weights for the Strong Coupling
Regime}

In the strong coupling regime, the coupling between voters induces
a pronounced tendency to vote alike. 
Contrary to the weak coupling regime, the optimal weights are not
necessarily uniquely determined. In fact, in many cases, the matrix $A=\lim_{N\to\infty}\left(\IE\left(\chi_{\nu}\chi_{\lambda}\right)\right)_{\nu,\lambda=1,\ldots,M}$ is singular so that
the limit of the linear system \eqref{eq:LESM} does not have a unique solution. We compute the matrix $A$ in Section \ref{strong_minimal_F} of the Appendix.

We next analyse the optimal weights in the three scenarios of the friendly world, the hostile world, and the split world, treated previously for the weak coupling regime in Section \ref{sec:weak}, under the assumption strong coupling.

\subsection{\label{subsec:Positive-Coupling-Low} Friendly World}

We start with coupling matrices $\bJ(J_{0},\jj)$ with $\jj >0 $ in the strong coupling regime. As discussed in
Section \ref{subsec:Basics}, the strong coupling regime is given by $J_{0}+(M-1)\jj>0 $. We also suppose that $\jj\leq J_{0}$ to ensure that
$\bJ(J_{0},\jj)$ is positive semi-definite.

Under this assumption, the matrix $A$ with $A_{\nu\lambda}=\lim_{N\to\infty}\IE(\chi_{\nu}\chi_{\lambda})$ is singular.

\begin{thm}\label{thm:rho1}
Suppose the coupling matrix $\bJ $ given by $\bJ(J_{0},\jj) $ with $\jj> 0$ is in the strong coupling regime. Then,
for all $\nu,\lambda $,
\begin{align}
   A_{\nu\lambda}&~=~\lim_{N\to\infty}\IE(\chi_{\nu}\chi_{\lambda})~=~1, \nonumber\\
b_{\lambda}&~=~\lim_{N\to\infty} \frac{1}{N}\,\IE\big(S\,\chi_{\lambda}\big)~
=~\sum_{\nu=1}^{M} \alpha_{\nu} \lim_{N\to\infty}\frac{1}{N_{\nu}}\,\IE\big(|S_{\nu}|\big)~=~\lim_{N\to\infty}\frac{1}{N}\,\IE\big(|S|\big).\label{eq:b}
\end{align}
\end{thm}
Recall that $S_{\lambda}=\sum_{i=1}^{N_{\lambda}}X_{\lambda i}$ and $S=\sum_{\lambda=1}^{M}S_{\nu} $.
The proof of this theorem can be found in Section \ref{subsec:Limit-Theorems}.

It follows that the asymptotically optimal weights are not uniquely determined:
\begin{thm}
\label{thm:pos_coupling_low}For a coupling matrix as in Theorem \ref{thm:rho1},
any $M$-tuple of positive weights is asymptotically optimal.
\end{thm}

\begin{proof}
This follows directly from the fact that $A=\boldsymbol{1}_{M\times M}$
and all entries of $b$ are equal.
\end{proof}
When voters of all groups are positively coupled, asymptotically,
the council votes will be almost surely unanimous by Proposition \ref{prop:orthants} and Theorem \ref{thm:rho1}.
Any distribution of weights among the groups gives rise to the same
council votes.

We contrast Theorem \ref{thm:pos_coupling_low} with a similar result for the collective bias model, Theorem 26 in \cite{KT2021_CBM}, which states that under a setup featuring strong correlation between votes belonging to different groups the optimal weights are indeterminate. This is a commonality between the MFM and the collective bias model: both admit scenarios where correlation between votes is so strong that the question of optimal weights becomes moot.

\subsection{\label{subsec:hostile_world_low} Hostile World}

The strong coupling regime in a hostile world leads to some complicated
yet interesting behaviour of the model.  If $M$ is even and all groups
are of the same size $\alpha_{\lambda}=1/M$, then there are $\left(\begin{array}{c}
M\\
M/2
\end{array}\right)$ points the vector of per capita voting margins $S_\lambda / N_\lambda$ assumes with positive (and equal) probability as $N$ goes to infinity (cf.\! Proposition \ref{prop:orthants}). Specifically,
these points are located in the orthants with precisely half the coordinates
positive and the other half negative. We can interpret this to mean
that, in a hostile world, we have ever-changing coalitions that maintain
the balance between the two alternatives being voted on. As all groups
are hostile toward each other, there are no permanent alliances as
in the other scenarios. As a consequence
of the shifting coalitions, the limit of the linear equation system \eqref{eq:LES}
has a unique solution as the matrix $\lim_{N\rightarrow \infty}A^N$ is not singular. The optimal
weights are given by $w=0$. As null weights lead to a council incapable
of reaching consensus on any proposal, this is yet another case where
in practice it is impossible to reach the minimal democracy deficit.

If $M$ is odd and the groups are of the same size, it is impossible
to achieve a perfect balance between the alternatives. Instead, the
closest possible approximation is realised, in which $\left(M+1\right)/2$
groups vote for one alternative and the rest vote for the other. Contrary
to $M$ even, here the optimal weights are unique and positive: $w_{\lambda}$ is proportional to $\frac{M+1}{M^{3}}$.
If we consider the limit of $M\rightarrow\infty$, the asymmetry disappears,
as a difference of one group in the council vote becomes insignificant.

The hostile world scenario illustrates that the optimal weights in
the strong coupling regime are uniquely determined in some cases aside from independent groups.

\subsection{\label{subsec:Two-Clusters-low} Split World}

Consider the coupling matrix with two clusters of groups first introduced
in Section \ref{subsec:Two-Clusters-high} with $J_{0}>\bar{J}$.
By Lemma \ref{lem:Matrix}, the strong coupling regime is equivalent to the condition $J_{0}+(M-1)\jj>1 $.

It follows from Proposition \ref{prop:orthants} that the vector of per capita voting margins
concentrates asymptotically in two orthants. However, contrary to the
friendly world scenario, these are not the positive and negative orthant. Rather, they are the two orthants where the coordinates belonging to each
cluster have the same sign and the two clusters are of opposite signs.

The optimal weights are not unique; however, contrary to the friendly world scenario with positive coupling between groups,
there is a condition on the total weight of the groups belonging
to each cluster:
\begin{thm}\label{thm:split_low}
\label{thm:two_cluster_weights-low}For a coupling matrix as in \eqref{eq:two_cluster_coupling} in the strong coupling regime, i.e.\!
with $J_{0}>\bar{J}>0$ and $J_{0}+(M-1)\jj>1 $,
any $M$-tuple of positive weights satisfying
\begin{equation}
\sum_{\lambda\in C_{1}}w_{\lambda}-\sum_{\lambda\in C_{2}}w_{\lambda}=\Theta\label{eq:two_clusters_weights_low}
\end{equation}
is optimal. The difference between the cluster weights $\Theta$ depends
on the parameters of the model.
\end{thm}

\begin{rem}
\label{rem:unique_minima}If the function $F$ defined in \eqref{eq:F}
has exactly two global minima, $m=\left(m_1,\ldots,m_M\right)$ located in the orthant with positive
coordinates $1,\ldots,M_{1}$ and negative coordinates $M_{1}+1,\ldots,M$,
and $-m$, then $\Theta=\sum_{\lambda\in C_{1}}\alpha_{\lambda}\left|m_{\lambda}\right|-\sum_{\lambda\in C_{2}}\alpha_{\lambda}\left|m_{\lambda}\right|$.
\end{rem}

\begin{proof}[Proof of Theorem \ref{thm:split_low}]
The statement follows from the observation that the matrix $A$ has block form
\[
A=\left(\begin{array}{cc}
\boldsymbol{1}_{M_{1}\times M_{1}} & -\boldsymbol{1}_{M_{1}\times M_{2}}\\
-\boldsymbol{1}_{M_{2}\times M_{1}} & \boldsymbol{1}_{M_{2}\times M_{2}}
\end{array}\right),
\]
and $b$ has identical entries for $\lambda \in C_{1}$ and the negative
of this value for $\lambda \in C_{2}$.
\end{proof}
The voters belonging to different clusters have a strong tendency
to vote opposite to each other. Asymptotically, the groups in cluster
1 will vote `yes' if and only if the groups in cluster 2 vote `no'
almost surely. Under the uniqueness assumption in Remark \ref{rem:unique_minima},
the absolute per capita voting margin $\IE\left(\left|S_{\lambda}\right|/N_{\lambda}\right)$
converges to the constant $m_{\lambda}$. Hence, we can interpret
$m_{\lambda}\in\left(0,1\right)$ as a measure of how large the typical
majority is, i.e.\! a measure of the cohesion within the group. As such, we
can interpret the terms $\sum_{\lambda\in C_{i}}\alpha_{\lambda}|m_{\lambda}|$
in the optimality condition \eqref{eq:two_clusters_weights_low} as
follows: a group contributes to the overall weight of its cluster
by being large and cohesive in its vote. Also, Theorem \ref{thm:two_cluster_weights-low}
makes no prescription as to how the joint weight of a cluster is to
be distributed among the groups. Similarly to the friendly world scenario, it is irrelevant how the weights are assigned among groups
that vote the same way almost surely.

\section{\label{sec:Independent-Clusters}Independent Clusters of Groups}

We have analysed both the weak and strong coupling regimes. The regime of the model determines the asymptotic behaviour
of the voting margins. In particular, it determines whether the group
voting margins are of order $\sqrt{N}$ or of order $N$. It is not possible
to have some voting margin $S_{\lambda}$ that grows like $\sqrt{N_{\lambda}}$
and some $S_{\nu}$ that behaves like $N_{\nu}$, unless the groups
are independent. If we posit $K\geq2$ clusters of groups which are
independent of each other, then these clusters can be in different
regimes. This assumption corresponds to a coupling matrix of block
form. Let $M_{1},\ldots,M_{K}$ be the number of groups in each cluster
$C_{1},\ldots,C_{K}$. Then the coupling matrix has the form
\begin{equation}
\mathbf{J}=\left(\begin{array}{cccc}
J^{1} & 0 & \ldots & 0\\
0 & J^{2} &  & \vdots\\
\vdots &  & \ddots & 0\\
0 & \ldots & 0 & J^{K}
\end{array}\right),\label{eq:indep_clusters_coupling}
\end{equation}
where $J^{i}$ is the $M_i\times M_i$ coupling matrix of cluster $C_{i}$. Let $A^{i}$
be $\left(\IE\left(\chi_{\lambda}\chi_{\nu}\right)\right)_{\lambda,\nu\in C_{i}}$
and $b^{i}=\left(b_{\lambda}\right)_{\lambda\in C_{i}}$. We will
call the $M_i$-vector of optimal weights for cluster $i$ $w^{i}$.
\begin{thm}
Let there be $K$ independent clusters with a coupling matrix as in
\eqref{eq:indep_clusters_coupling}. 
Then, for all clusters $i=1,\ldots,K$, the optimal weights for all
groups $\lambda\in C_{i}$ are given by the linear equation system
\[
A^{i}w^{i}=b^{i}.
\]
\end{thm}

An immediate consequence of this theorem is that if there are two
independent clusters, the first in the weak coupling regime, the
second in the strong coupling regime, then the second cluster
will receive all the weight as the overall population goes to infinity.
\begin{cor}
Let $K=2$ and let the first cluster be in the weak and the second
in the strong coupling regime. Also assume that all entries of
$J^{2}$ are non-negative. Then the total weight of cluster $1$ is $O\left(1/\sqrt{N}\right)$,
and the total weight of cluster $2$ is a positive constant.
\end{cor}

This corollary illustrates that clusters in the weak coupling regime,
whose voters interact loosely with each other, receive little weight
compared to clusters in the strong coupling regime. Why does this
happen? We have to think about whose opinion the representative of
a group represents in the council. The answer is they only represent
the difference in votes between the alternative that won, `yes'
or `no', and the alternative that lost, the absolute voting margin.
Hence, on average, the representatives cast their vote in the council
in the name of a number of people in their group that corresponds
to the expected absolute voting margin in their group. The expected
per capita absolute voting margin in the weak coupling regime behaves
like $1/\sqrt{N}$, whereas in the strong coupling regime it converges to a positive
constant as $N$ goes to infinity. That is the reason the strong coupling
regime representatives should receive more weight in the council:
they stand for more people in favour of or against the proposal.

We also see that if there is a cluster of a single group, then that
group will either have a weight proportional to the square root of
its population if it is in the weak coupling regime, or a weight
proportional to its population if it is in the strong coupling
regime. Hence, we recover the previous results found in \cite{SRL,KL1,Langner}.

\section{Conclusion\label{sec:conclusion}}

We used a multi-group MFM to study the problem of optimal weights which minimise the democracy deficit in a two-tier voting system. This model is a generalisation both of impartial culture and the classical single-group version of the MFM. It allows us to study correlated voting across group boundaries and how this correlation affects the optimal weights.

In Section \ref{sec:weak}, we studied the optimal weights under weak coupling between voters. The optimal weights are given by the sum of a constant independent of the group's size and a term proportional to the square root of each group's population. This result is a generalisation of Penrose's square root law. An interesting aspect is that the sign of the constant term in the formula for the optimal weights depends on the specific structure of the coupling matrix which describes coupling between voters belonging to different groups. In the friendly world scenario, where all groups interact positively with each other, the constant is positive. In other scenarios, such as  the hostile world, where all groups are antagonistic to each other, the constant is negative. As a direct consequence, very small groups are assigned a negative optimal weight. Finally, we saw in the split world scenario that the constant's sign can be different for different groups. In a split world, generally speaking, the groups belonging to one cluster will have a positive constant and the others, a negative one.

In Section \ref{sec:Optimal-Weights-Low}, we examined the optimal weights under strong coupling between voters. We found that in some cases, such as the friendly and split world scenarios, the optimal weights are indeterminate. In the hostile world scenario, however, optimal weights can be uniquely determined under some circumstances, rounding out the picture of the strong coupling regime, which differs considerably from the previous results obtained for independent groups.

Finally, we studied a scenario in Section \ref{sec:Independent-Clusters} with several independent clusters of groups. This generalises the independent groups case in the sense that here the independent sets of voters comprise more than just a single group each. We found that the optimal weights favour those clusters which feature strong coupling between their voters at the expense of clusters with weak coupling.

In our opinion, the results of this paper are interesting from a theoretical point of view, as they explore the impact of interaction among voters both inside their group and across group borders in two-tier voting systems. However, these results may also be of practical use in designing voting systems. A careful statistical analysis of correlations between voters may decide which model is appropriate in a given situation.

As a rule, systems with a long tradition of cooperation will presumably be better modelled by a collective bias voting model, while more loosely organised and more recently founded systems might behave more like an MFM. The former is most likely the case for federal states, like the US, in particular for the Electoral College. The latter model seems to be more suitable for a confederation of independent states like the EU, in particular for the Council of Ministers.

An alternative interpretation of the weak coupling regime with positive correlation between two groups and the scenario of a positive correlation in a collective bias model is that of order vs.\! disorder: in the weak coupling regime of the MFM, not only is there weaker positive correlation between the groups, but also the internal cohesion within each of the groups is much weaker. So the MFM would be more apt to model chaotic situations in which votes are close even within each group, whereas the collective bias model is better for situations with stable majorities describing more ordered situations.

\section*{Acknowledgments} We are grateful to the anonymous referees for many valuable suggestions.

\appendix

\section{\label{sec:Appendix}Appendix}

\subsection{\label{subsec:more_dem_def}More on the Democracy Deficit and the Optimal Weights}

A basic result concerning weighted voting systems is that if we multiply each weight by the same positive constant and keep
the relative quota $q$ fixed, we obtain an equivalent voting system.
If the weights $w_{\lambda}$ minimise the democracy deficit $\Delta_{1}$,
then the (equivalent) weights $w_{\lambda}/\sigma$ for any $\sigma>0$
minimise the `renormalised' democracy deficit $\Delta_{\sigma}$
defined by
\begin{align*}
\Delta_{\sigma}~=~\Delta_{\sigma}(v_{1},\ldots,v_{M})~:=~\mathbb{E}\left[\left(\frac{S}{\sigma}-\sum_{\lambda=1}^{M}v_{\lambda}\chi_{\lambda}\right)^{2}\right]\,.
\end{align*}

Recall that whenever we speak of the uniqueness of the vector of optimal weights,
it shall be understood to mean `uniqueness up to multiplication by
a positive constant.'

It is, therefore, irrelevant whether we minimise $\Delta_{1}$ or
$\Delta_{\sigma}$ as long as $\sigma>0$. In this article, we
compute optimal weights as $N$ tends to infinity. As a rule, in this
limit, the minimising weights for $\Delta_{1}$ will also tend to infinity.
It is therefore useful to minimise $\Delta_{\sigma}$ with an $N$-dependent
$\sigma$ to keep the weights bounded. For the MFM, the two possible
choices for $\sigma$ turn out to be $\sqrt{N}$ and $N$. Which
one of these is appropriate depends on the parameters of the model
(see Section \ref{subsec:Regimes}): in the weak coupling regime, we choose $\sigma = \sqrt{N}$ and in the strong coupling regime $\sigma = N$. Using this normalisation by $\sigma$ is how we obtain optimal weights that asymptotically converge to constants instead of diverging to infinity as the population goes to infinity. Thus, the formulas in Sections \ref{sec:weak} and
\ref{sec:Optimal-Weights-Low} feature the asymptotic relative group sizes $\alpha_\lambda$ instead of the absolute group sizes $N_\lambda$.

%

Recall the definition of the linear equation system \eqref{eq:LESM}. The matrix $A^N$ is invertible under rather
mild conditions.
\begin{defn}
\label{def:suffrand} We say that a voting measure $\IP$ on $\prod_{\lambda=1}^{M}\,\{-1,1\}^{N_{\lambda}}$
is \emph{sufficiently random} if
\begin{align}
\IP\,(\chi_{1}=s_{1},\ldots,\chi_{M}=s_{M})~>~0\qquad\text{for all }s_{1},\ldots,s_{M}\in\{-1,1\}.\label{eq:suffrand}
\end{align}
\end{defn}

Note that \eqref{eq:suffrand} is not very restrictive. For example,
if the support ${\rm supp\,\IP}$ of the voting measure $\IP$ is the whole
space $\{-1,1\}^{N}$, then $\IP$ satisfies \eqref{eq:suffrand}.
As a matter of fact, all versions of the MFM studied in this article
are sufficiently random for finite $N$. However, asymptotically
this property is lost in some cases, meaning the limiting distribution is no longer sufficiently random.
\begin{prop}
\label{prop:posdef} Let $\IP$ be a voting measure and let $A^N$ be
defined by \eqref{eq:A0}.
\begin{enumerate}
\item The matrix $A^N$ is positive semi-definite.
\item $A^N$ is positive definite if $\IP$ is sufficiently random.
\end{enumerate}
\end{prop}

\begin{proof}
This is Proposition 12 in \cite{KT2021_CBM}.
\end{proof}

The next theorem immediately follows from the previous proposition.
\begin{thm}
\label{thm:weights} If the voting measure $\IP$ is sufficiently
random, the optimal weights minimising the democracy deficit $\Delta_{\sigma}$
are unique and given by
\begin{align}
w^N~=~\left(A^N\right)^{-1}\,b^N.\label{eq:w}
\end{align}
\end{thm}

Although for finite $N$ typical voting measures are sufficiently random, including the MFM, the above result
is of rather limited usability as it is practically impossible to compute
the ingredients like $A^N=\IE(\chi_{\lambda}\chi_{\nu})$ and $\IE(S\chi_{\lambda})$
for finite (but fairly large) $N$.

In the following, we shall compute these quantities approximately
for $N\to\infty$ .
More precisely, for each $N=N_{1}+\cdots+N_{M}$, we define voting measures $\IP_{N}$ as well as
the derived quantities $A^{N}_{\nu\lambda}=\IE_{N}(\chi_{\nu}\chi_{\lambda})$, $(b^{N})_{\nu}=\frac{1}{\sigma}\IE(\chi_{\nu}S)$ and weights $(w^N_{\nu})_{\nu}$ and then evaluate their limits as $N\rightarrow \infty$.

In the following discussion, we assume that the limits $A:=\lim A^{N}$ and $b:=\lim b^{N}$ exist. This assumption is fulfilled in the models we discuss in this paper.

Even if the matrices $A^{N}$ are invertible for each $N$, the limit $A=\lim_{N\to\infty} A^{N}$ may be singular. If the limit matrix $A$ \emph{is} invertible, then the weights
\begin{align}
   w^{N}~&=~\left(A^{N}\right)^{-1}\,b^{N}\nonumber\\
\intertext{converge to}
   w~&=~A^{-1}\,b,\label{eq:winf}
\end{align}
the optimal weight for the limiting (i.e.\! large $N$) distribution of the model.
In these cases, we compute the weights $w$ and use them as approximations for the optimal weights $w^{N}$ for large $N$.

If the limit matrix $A$ is not invertible, then the equation
\begin{align*}
   A\,w~=~b
\end{align*}
has either no solution at all, or the solutions form a whole affine subspace $\mathcal{W}$ (of dimension at least $1$).
In the latter case, for all $w\in\mathcal{W}$,
\begin{align*}
   \left|A^{N}\,w- b^N\right|~\to~0\,.
\end{align*}
We shall then say that $\mathcal{W}$ consists of approximate solutions for large $N$. The (approximately) optimal weights are  \emph{not unique} in this case. In the cases considered in this article, $\mathcal{W}$ is typically of codimension $1$.

To compute optimal weights according to \eqref{eq:winf}, we have to evaluate the matrix $A$ given by
\begin{align}\label{eq:A}
   A_{\lambda\nu}~=~\lim_{N\to\infty}\,\IE\big(\chi_{\lambda}\chi_{\nu}\big)
\end{align}
as well as the inverse of $A$. We also need to compute
\begin{align}\label{eq:b0}
   b_{\nu}~=~\lim_{N\to\infty}\,\frac{1}{\sigma_{N}}\IE(\chi_{\nu}S)\,.
\end{align}

The quantities $\chi_{\nu}$ and $S$ depend on the voting margins
\begin{align}
\boldsymbol{S} & =\left(S_{1},\ldots,S_{M}\right)=\left(\sum_{i_{1}=1}^{N_{1}}X_{1i_{1}},\ldots,\sum_{i_{M}=1}^{N_{M}}X_{1i_{M}}\right)\label{eq:norm_sums_CWM}
\,.\end{align}
Therefore, we have to understand the large-$N$-behaviour of $\boldsymbol{S}$ in order to compute the limits \eqref{eq:A} and \eqref{eq:b0}.

\subsection{Regimes of the Mean-Field Model}\label{subsec:Regimes}

In this section, we will discuss the patterns of behaviour of the voting margins depending
on the parameters of the model, that is on the strength of the coupling between the voters.

Let for $x\in \IR^d$, $d\in \IN$, the symbol $\delta_x$ mean the Dirac measure at the point $x$. By $\mathcal{N}\left(0,\sigma^2\right)$ with $\sigma >0$, we refer to the normal distribution with mean 0 and variance $\sigma^2$, and the symbol $\dto $ means convergence in distribution as $N$ goes to infinity. For a positive definite matrix $C$, $\mathcal{N}(0,C)$ stands for the centred multivariate normal distribution with covariance matrix $C$.

It is well known that the single-group MFM defined in equation \eqref{eq:sCW} has a `phase transition' at $J=1$,
i.e.\! the behaviour of $\sum_i X_{i}$ is qualitatively different below and above this threshold.
More precisely, in the single group model we have that
\begin{align*}
   \frac{1}{N}\sum_{i=1}^{N}X_{i}~\dto~\delta_{0},\qquad
\frac{1}{\sqrt{N}}\sum_{i=1}^{N}X_{i}~\dto~\mathcal{N}\left(0,(1-J)^{-1}\right)
\end{align*}
as long as $J<1$, but
\begin{align*}
   \frac{1}{N}\sum_{i=1}^{N}X_{i}~\dto~\ed{2}\big(\delta_{-m}+\delta_{m}\big),
\end{align*}
with a $J$-dependent $m>0 $ for $J>1$
 (see e.g.\! \cite{Ellis} or \cite{K-MS} for a more elementary proof).

For the multi-group models introduced in Definition \ref{multi-group_MFM}, we define:
\begin{defn}\label{def:regimes}
 We say that the MFM with coupling matrix $\bJ $ is in the \emph{weak coupling regime}, if $\bJ<\bI$. Here $\bI $ is the $M\times M $-identity matrix and $\bJ<\bI$ means that $\bI-\bJ $ is positive definite.

We say that the MFM is in the \emph{critical regime} if $\bI-\bJ $ is positive semi-definite but not positive definite.

We say that the MFM is in the \emph{strong coupling regime} if $\bI-\bJ $ is not positive semi-definite.
\end{defn}
Of the three regimes which make up the parameter space of the MFM, the critical regime is by far the smallest. We will exclusively deal with the other two regimes which are of
more practical importance.

In partial analogy to the single group case, we have
\begin{thm}\label{thm:clt}
   Suppose $\bJ$ is either positive definite or $\bJ$ is a homogeneous coupling matrix (see Definition \ref{def:hom}).

If $\bJ$ is in the weak coupling regime, then
\begin{align*}
   &\left(\frac{S_1}{N_1},\ldots,\frac{S_M}{N_M}\right)~\dto~\delta_{0},&&
\left(\frac{S_1}{\sqrt{N_1}},\ldots,\frac{S_M}{\sqrt{N_M}}\right)~\dto~\mathcal{N}\left(0,(\mathbf{I}-\bJ)^{-1}\right).
\end{align*}

If $\bJ $ is a homogeneous coupling matrix and $\bJ $ is in the strong coupling regime, then
\begin{align}\label{eq:stronghom}
  \left(\frac{S_1}{N_1},\ldots,\frac{S_M}{N_M}\right)~\dto~\ed{2}\big(\delta_{-\mathbf{m}}+\delta_{\mathbf{m}}\big),
\end{align}
where $\mathbf{m}$ is an $M $-dimensional vector with strictly positive entries.
\end{thm}
For a proof of this result see \cite{Fedele}, \cite{Loewe}, or \cite{KT2021_CWM}.

We remark that the vector $\mathbf{m}$ has the form $\mathbf{m}=(m,m,\ldots,m)$, where $m$ is the positive solution of equation \eqref{eq:CW}.

Theorem \ref{thm:clt} has the following important consequence:
\begin{thm}\label{thm:invert}
 If $\bJ$ is in the weak coupling regime, then the limit
\begin{align*}
   A_{\nu\lambda}~:=~\lim_{N\to\infty}A^{N}_{\nu\lambda}~=~\lim_{N\to\infty}\IE(\chi_{\nu}\chi_{\lambda})
\end{align*}
exists, and the matrix $A$ is positive definite hence invertible.
\end{thm}
\begin{proof}
   We set
\begin{align*}
   \chi(x)~&:=~\left\{
                \begin{array}{ll}
                  1, & \hbox{if $x>0$,} \\
                  -1, & \hbox{otherwise,}
                \end{array}
              \right.
\intertext{and}
  \chi_{\nu}^{N}~&:=~\chi\left(\ed{\sqrt{N_{\nu}}}S_{\nu}\right).
\end{align*}
The matrix $\left(\mathcal{X}^{N}_{\nu\lambda}\right) \coloneq \left(\chi_{\nu}^{N}\chi_{\lambda}^{N}\right)$ is positive semi-definite and
we have $A^{N}=\IE\left(\mathcal{X}^{N}\right)$. Set $\mathcal{X}:=\lim_{N\rightarrow \infty} \mathcal{X}^{N}$.

By Theorem \ref{thm:clt}, the vectors $\chi^{N}=\left(\chi_{1}^{N},\ldots,\chi_{M}^{N}\right)$ converge in distribution to $\chi(Z)=\left(\chi(Z_{1}),\ldots,\chi(Z_{M})\right)$, where the distribution of $Z=(Z_{1},\ldots,Z_{M})$ is an $M$-dimensional centred normal distribution with covariance matrix $(\mathbf{I}-\bJ)^{-1}$.

So if $Q$ denotes the distribution of $\chi(Z) $, we have $\supp\,Q=\{ -1,1 \}^{M}$.

Now suppose that $\langle x,A x \rangle=\lim_{N\to\infty} \left< x,A^{N} x \right> = 0$. Then
\begin{align*}
   \IE\Big(\langle x, \mathcal{X} x \rangle\Big)=0.
\end{align*}
Since $\langle x, \mathcal{X} x \rangle \geq 0$ and the random matrices $\mathcal{X} $ and  $\chi(Z) \chi(Z)^T$ are identically distributed, it follows that
\begin{align*}
   \sum_{\nu=1}^{M} \chi(Z_{\nu})\, x_{\nu} = 0 \quad Q\textup{-almost surely.}
\end{align*}

Since $\{ -1,1 \}^{M}$ spans $\IR^{M}$, we obtain $x=0$. Thus,
the matrix $A$ is positive definite.

\end{proof}

The entries of the matrix $A$ can be expressed by the entries of the covariance matrix $C=\left(C_{\nu\lambda}\right)_{\nu,\lambda=1,\ldots,M}=(\mathbf{I}-\bJ)^{-1}$.

\begin{prop}\label{prop:arcsin}
   If $\bJ$ is in the weak coupling regime, then
\begin{align*}
   A_{\nu\lambda}~=~\IE(\chi_{\nu}\chi_{\lambda})~=~\frac{2}{\pi}\,\arcsin \left( \frac{C_{\nu\lambda}}{\sqrt{C_{\nu\nu}}\sqrt{C_{\lambda\lambda}}} \right) \,.
\end{align*}
In particular, $A_{\nu\nu}=1 $ and $-1<A_{\nu\lambda}<1$ for $\nu\not=\lambda $.
\end{prop}
Proposition \ref{prop:arcsin} is proved in Section \ref{ssec:linearS}.

The mathematical properties of the multi-group MFMs are intimately connected to the function
\begin{align}
   F(x)~:=~\frac{1}{2}\,x^{T}\,\sqrt{\alpha}\,\bJ^{-1}\,\sqrt{\alpha}\,x-\sum_{\lambda=1}^{M}\alpha_{\lambda}\ln\cosh x_{\lambda},\quad x\in\IR^{M}\,.\label{eq:F}
\end{align}
 In the formula above, the $M\times M$ matrix $\sqrt{\alpha}$
is diagonal with entries $\sqrt{\alpha_{\lambda}}$ on the diagonal.

We recall a few facts about this connection, details can be found in \cite{KT2021_CWM}.
By $P_{t}, t\in[-1,1] $, we denote the probability measure on $\{ -1,1 \} $ defined by
\begin{align}\label{prod_measure}
   P_{t}(1)~=~\ed{2}\,(1+t), \qquad P_{t}(-1)~=~\ed{2}\,(1-t)\,.
\end{align}
By $P_{t}^{\otimes n} $, we refer to the corresponding product measure on $\{ -1,1 \}^{n} $, and by $E_{t}^{\otimes n} $, to the associated expectation.

For a function $f:\prod_{\nu=1}^{M}\{ -1,1 \}^{N_{\nu}}\to \IR$, we define a function $\widetilde{f}:[-1,1]^{M}\to\IR $ by
\begin{align*}
   \widetilde{f}(x_{1},\ldots,x_{M})~:=~E_{x_{1}}^{\otimes N_{1}}\cdots E_{x_{M}}^{\otimes N_{M}}\Big(f\left(X_{11},\ldots,X_{1N_{1}},~\ldots~,X_{M1},\ldots,X_{MN_{M}}\right)\Big).
\end{align*}
Note that in the above formula the random variables $X_{\nu i}$ are independent \emph{with respect to} $\prod_{\nu=1}^{M} P_{x_{\nu}}^{\otimes N_{\nu}}$.
Now we set
\begin{align*}
   Z_{N}(f)~&:=~\int_{\IR^{M}}\,\widetilde{f}\big(\tanh\left(x_{1}\right),\ldots,\tanh\left(x_{M}\right)\big)\;e^{-NF(x)}\,\textup{d}x,\\
Z_{N}~&:=~\int_{\IR^{M}}\;e^{-NF(x)}\,\textup{d}x.
\end{align*}
Finally, we can evaluate expectations of the random variables of the MFM with positive definite coupling matrix $\bJ$ (see \cite{KT2021_CWM}):
\begin{align}\label{eq:hubstra}
   \IE_{\bJ}\big(f\left(X_{11},\ldots,X_{1N_{1}M},~\ldots~,X_{M1},\ldots,X_{MN_{M}}\right)\big)~\approx~\frac{1}{Z_{N}}\,Z_{N}(f)\quad\text{as } N\to\infty.
\end{align}
Above, we used the symbol `$\approx$' in the following sense:
\begin{defn}\label{def:appr}
Real-valued sequences $f_{N},g_{N}$ are called \emph{asymptotically equal} (as $N\to\infty$), in short $f_{N}\approx g_{N}$, if
 \begin{align*}
    \lim_{N\rightarrow\infty}\frac{f_{N}}{g_{N}}~=~1.
 \end{align*}
\end{defn}

In \cite{KT2021_CWM}, formula \eqref{eq:hubstra} was used to show Theorem \ref{thm:clt} by using Laplace's Theorem to evaluate
the right hand side of \eqref{eq:hubstra} asymptotically. Laplace's Theorem relates the behaviour of such integrals to the
minima of the function $F$. It turns out that the weak coupling regime is given by those $\bJ$ for which $F$ has a unique
non-degenerate minimum at $x=0$. In the strong coupling regime, $F$ assumes its minima away from the origin. Due to the symmetry of $F$, there are always at least two minima in this case.

To our knowledge, there is currently no general result concerning the minima of $F$ for the strong coupling regime. This is the reason why Theorem \ref{thm:clt} has results for the strong coupling regime only for homogeneous coupling matrices.

\subsection{\label{subsec:Matrix} Some Linear Algebra}

In this section, we analyse the eigenvalues of certain matrices that appear in the analysis of the MFM and the optimal weights. The knowledge of these eigenvalues leads to bounds on the parameters for positive definiteness of the coupling matrix $\bJ$ and the two regimes.
Suppose $a,b\in\IR$ and $M_{1}, M_{2}\in\IN$ with $M_{1}\geq 2, M_{2}\geq 0$, and set $M=M_{1}+M_{2}$.

We introduce the following notation for coupling matrices as in \eqref{eq:two_cluster_coupling}:
\begin{defn} \label{block_matrix}
   By $\bJ_{M_{1},M_{2}}(J_{0},\jj) $ we denote the matrix
\begin{align*}
   \bJ_{M_{1},M_{2}}(J_{0},\jj)_{\nu\lambda}~=~\left\{
                                    \begin{array}{ll}
                                      J_{0}, & \hbox{if $\nu=\lambda $,} \\
                                      \jj, & \hbox{if $\nu\not=\lambda$ and $\nu,\lambda\leq M_{1} $ ,} \\
                                      \jj, & \hbox{if $\nu\not=\lambda$ and $\nu,\lambda> M_{1} $,} \\
                                      -\jj, & \hbox{otherwise.}
                                    \end{array}
                                  \right.
\end{align*}
\end{defn}

We define the $M\times M$-matrix $A=\bJ_{M_{1},M_{2}}(b,a)$.

\begin{lem}\label{lem:Matrix}
   The matrix $A$ has eigenvalues $b+(M-1)a$ and $b-a$.
$A$ is positive definite if and only if
\begin{align*}
   a~<~b \qquad \text{and} \qquad -a~<~\ed{M-1}\;b.
\end{align*}
\end{lem}
\begin{proof}
   Define the vector $w$ by
\begin{align*}
   w_{i}~&=~\left\{
                   \begin{array}{ll}
                     1, & \hbox{for $i\leq M_{1}$,} \\
                     -1, & \hbox{for $i>M_{1}$.}
                   \end{array}
                 \right.
\end{align*}
Then $w$ is an eigenvector for the eigenvalue $b+(M-1)a$.

For $k\in \{1,\ldots, M_{1}-1, M_{1}+1,\ldots,M-1\}$, we set
\begin{align*}
   v^{k}_{i}~&=~\left\{
                   \begin{array}{ll}
                     1, & \hbox{for $i=k$,} \\
                     -1, & \hbox{for $i=k+1$,}\\
                      0, & \hbox{otherwise;}
                   \end{array}
                 \right.
\end{align*}
and for $k=M_{1}$, we set
\begin{align*}
   v^{M_{1}}_{i}~&=~\left\{
                   \begin{array}{ll}
                     1, & \hbox{for $i=M_{1}$,} \\
                     1, & \hbox{for $i=M_{1}+1$,}\\
                      0, & \hbox{otherwise.}
                   \end{array}
                 \right.
\end{align*}
The vectors $v^{k},\, k=1,\ldots ,M-1,$ are eigenvectors for the eigenvalue $b-a$.
\end{proof}

If $a\not=b$, then $b+(M-1)a$ is a simple eigenvalue and $b-a$ is $(M-1)$-fold degenerate.

We remark that the matrix $A$ can be written as
\begin{align*}
   A~=~(b-a)\,I\;+\;a\,|w\rangle\langle w|,
\end{align*}
where $|w\rangle\langle w|  $ denotes the orthogonal projection onto the vector $w$.

Next we calculate the inverse matrix of $A$.

\begin{lem}\label{lem:Ainv}
   If  $a\not=b$ and $-a\not=\ed{M-1}b$, then the matrix $A=\bJ_{M_{1},M_{2}}(b,a)$ is invertible and
\begin{align*}
   A^{-1}_{i j}~=~\ed{(b-a)(b+(M-1)\,a)\,}\;\left\{
                \begin{array}{ll}
                  b+(M-2)a, & \hbox{for $i=j$,} \\
                  -a, & \hbox{for $i\not=j$ and ($i,j\leq M_{1}$ or $i,j>M_{1}$}), \\
                  a, & \hbox{otherwise.}
                \end{array}
              \right.
\end{align*}

\end{lem}
\begin{proof}
   The proof is a lengthy but straightforward computation.
\end{proof}

The second type of matrix we deal with in Section \ref{sec:weak} is the balanced
coupling matrix $\bJ(b,a)$. Its positive definiteness is characterised
by
\begin{lem}\label{lem:balanced_matrix}
The matrix $\bJ(b,a)$ is positive definite if and only if
\[
a<b\quad and\quad-a<\frac{1}{M-1}b.
\]
\end{lem}

\begin{proof}
This lemma can be proved analogously to Lemma \ref{lem:Matrix}. The key observation
is that $\bJ(b,a)$ has the eigenvalues $b-a$ and $b+(M-1)a$.
\end{proof}

\subsection{Entries of the Linear Equation System \eqref{eq:LES}}\label{ssec:linearS}

In order to calculate the optimal weights, we first need the general
form of the entries in the matrix $A$ in \eqref{eq:A} and \eqref{eq:A0}.
\begin{lem}
\label{lem:entries_A}The entries of $A$ are $A_{\lambda\mu}\approx4\,\mathbb{P}\left(S_{\lambda},S_{\nu}>0\right)-1$
for all $\lambda,\mu=1,\ldots,M$.
\end{lem}

\begin{proof}
This is a straightforward calculation.
\end{proof}
We show that
\begin{prop}
\label{prop:A_b_high}Let $C=\left(C_{\kappa\nu}\right)_{\kappa,\nu=1,\ldots,M}=(\bI-\bJ)^{-1}$
be the covariance matrix defined in Theorem \ref{thm:clt}. In the
weak coupling regime, we have
\begin{enumerate}
\item $\mathbb{E}(\chi_{\kappa}\chi_{\nu})\approx\frac{2}{\pi}\arcsin\left(\frac{C_{\kappa\nu}}{\sqrt{C_{\kappa\kappa}}\sqrt{C_{\nu\nu}}}\right)$
for all $\kappa\neq\nu$,
\item $\mathbb{E}(\chi_{\kappa}S_{\kappa})\approx\sqrt{\frac{2C_{\kappa\kappa}}{\pi}N_{\kappa}}$
for all $\kappa$,
\item $\mathbb{E}(\chi_{\kappa}S_{\nu})\approx\sqrt{\frac{2}{\pi C_{\kappa\kappa}}N_{\nu}}C_{\kappa\nu}$
for all $\kappa\neq\nu$.
\end{enumerate}
\end{prop}

\begin{proof}
By Lemma \ref{lem:entries_A}, we have
\begin{align*}
\mathbb{E}(\chi_{\kappa}\chi_{\nu}) & \approx4\,\mathbb{P}\left(\frac{S_{\kappa}}{\sqrt{N_{\kappa}}},\frac{S_{\nu}}{\sqrt{N_{\nu}}}>0\right)-1.
\end{align*}
We need to calculate the two-dimensional marginal distribution of
$\left(\frac{S_{\kappa}}{\sqrt{N_{\kappa}}},\frac{S_{\nu}}{\sqrt{N_{\nu}}}\right)$.
This distribution is bivariate normal with mean 0 and covariance matrix
\[
\left(\begin{array}{cc}
C_{\kappa\kappa} & C_{\kappa\nu}\\
C_{\kappa\nu} & C_{\nu\nu}
\end{array}\right).
\]
For convenience sake, we set $X':=\frac{S_{\kappa}}{\sqrt{N_{\kappa}}}$
and $Y':=\frac{S_{\nu}}{\sqrt{N_{\nu}}}$. We standardise by dividing
by the standard deviations:
\begin{align*}
X & :=\frac{X'}{\sqrt{C_{\kappa\kappa}}},\quad Y:=\frac{Y'}{\sqrt{C_{\nu\nu}}},
\end{align*}
so that both $X$ and $Y$ have marginal standard normal distributions.
The correlation between them is given by
\begin{align*}
\rho & :=\mathbb{E}(XY)=\frac{\mathbb{E}(X'Y')}{\sqrt{C_{\kappa\kappa}}\sqrt{C_{\nu\nu}}}=\frac{C_{\kappa\nu}}{\sqrt{C_{\kappa\kappa}}\sqrt{C_{\nu\nu}}}.
\end{align*}
We set
\[
Z:=\frac{Y-\rho X}{\sqrt{1-\rho^{2}}}
\]
and note that $X$ and $Z$ are independent: $X$ and $Z$ are both
normal and their covariance is
\begin{align*}
\mathbb{E}(XZ) & =\frac{\mathbb{E}(XY)-\rho\mathbb{E}\left(X^{2}\right)}{\sqrt{1-\rho^{2}}}=\frac{\rho-\rho}{\sqrt{1-\rho^{2}}}=0.
\end{align*}
It is easily verified that the distribution of $Z$ is standard normal.
We let $\phi$ represent the density function of the standard normal
distribution and calculate
\begin{align*}
\mathbb{P}(X',Y'>0) & =\mathbb{P}(X,Y>0)=\mathbb{P}\left(X>0,Z>\frac{-\rho X}{\sqrt{1-\rho^{2}}}\right)\\
 & =\intop_{0}^{\infty}\phi(x)\intop_{\frac{-\rho x}{\sqrt{1-\rho^{2}}}}^{\infty}\phi(z)\,\text{d}z\,\text{d}x=\frac{1}{2\pi}\intop_{0}^{\infty}\intop_{\frac{-\rho x}{\sqrt{1-\rho^{2}}}}^{\infty}e^{-\frac{x^{2}+z^{2}}{2}}\,\text{d}z\,\text{d}x.
\end{align*}
We switch to polar coordinates and the last integral above equals
\begin{align*}
\frac{1}{2\pi}\intop_{0}^{\infty}\intop_{\arctan\frac{-\rho}{\sqrt{1-\rho^{2}}}}^{\pi/2}e^{-\frac{r^{2}}{2}}\,r\,\text{d}\varphi\,\text{d}r & =\frac{1}{4}+\frac{1}{2\pi}\arcsin(\rho).
\end{align*}
We next show the third result and note that the second one is a special
case of the third. We set $X:=\frac{S_{\kappa}}{\sqrt{N_{\kappa}}}$
and $Y:=\frac{S_{\nu}}{\sqrt{N_{\nu}}}$ and use the conditional expectation
\[
\mathbb{E}(Y|X)=\frac{C_{\kappa\nu}}{C_{\kappa\kappa}}X,
\]
which can be easily verified (for a proof see Chapter 4 of \cite{BT}).
Let $\textup{sgn}(x)$ stand for the sign of $x\in \IR$ and $\II A$, for any measurable set A, for the indicator function of $A$.
We are interested in $\mathbb{E}\left(\chi_{\kappa}\frac{S_{\nu}}{\sqrt{N_{\nu}}}\right)$,
which is equal to $\mathbb{E}(\text{sgn}(X)Y)$, therefore, we need
to calculate $\mathbb{E}(Y\II\{X>0\})$ and $\mathbb{E}(Y\II\{X<0\})$.
Their difference is the expectation we are looking for.
\begin{align*}
\mathbb{E}(Y\II\{X>0\}) & =\int\int\II\{X>0\}Y\mathbb{P}^{X,Y}(\text{d}x,\text{d}y)=\int\II\{X>0\}\int Y\mathbb{P}^{Y|X}(\text{d}y)\mathbb{P}^{X}(\text{d}x)\\
 & =\int\II\{X>0\}\mathbb{E}(Y|X=x)\mathbb{P}^{X}(\text{d}x)=\intop_{0}^{\infty}\frac{c_{\kappa\nu}}{C_{\kappa\kappa}}x\frac{1}{\sqrt{2\pi C_{\kappa\kappa}}}e^{-\frac{x^{2}}{2C_{\kappa\kappa}}}\text{d}x\\
 & =\frac{C_{\kappa\nu}}{\sqrt{2\pi C_{\kappa\kappa}}}.
\end{align*}
A very similar calculation yields
\[
\mathbb{E}(Y\II {\{X<0\}})=-\frac{C_{\kappa\nu}}{\sqrt{2\pi C_{\kappa\kappa}}}.
\]
Therefore, we have
\[
\mathbb{E}\left(\chi_{\kappa}\frac{S_{\nu}}{\sqrt{N_{\nu}}}\right)=\frac{\sqrt{2}C_{\kappa\nu}}{\sqrt{\pi C_{\kappa\kappa}}}.
\]
\end{proof}
\begin{cor}
Let $C=\left(C_{\kappa\nu}\right)_{\kappa,\nu=1,\ldots,M}$ be the
covariance matrix defined in Theorem \ref{thm:clt}. In the weak coupling
regime, the linear equation system \eqref{eq:LES} reads
\begin{equation*}
\left(\frac{2}{\pi}\arcsin\left(\frac{C_{\kappa\nu}}{\sqrt{C_{\kappa\kappa}}\sqrt{C_{\nu\nu}}}\right)\right)_{\kappa,\nu=1,\ldots,M}w=\sqrt{\frac{2}{\pi}}\left(\sqrt{C_{\kappa\kappa}}\sqrt{\alpha_{\kappa}}+\sum_{\nu\neq\kappa}\frac{C_{\kappa\nu}}{\sqrt{C_{\kappa\kappa}}}\sqrt{\alpha_{\nu}}\right)_{\kappa=1,\ldots,M}.
\end{equation*}
\end{cor}

We next show that the linear equation system \eqref{eq:LES} has a
unique solution in the weak coupling regime.
\begin{prop}
\label{prop:nonsingular}The matrix $A=\lim_{N\to\infty}\mathbb{E}(\chi_{\kappa}\chi_{\nu})_{\kappa,\nu=1,\ldots,M}$
is non-singular in the weak coupling regime.
\end{prop}

\begin{proof}
The covariance matrix $C=\bI-\bJ$ is positive definite. Thus, the limiting
distribution is sufficiently random, and by Proposition \ref{prop:posdef}
the claim follows.
\end{proof}

\subsection{\label{subsec:rho}Proof of Proposition \ref{prop:rho}}
Observe that in this case $\rho=\lim_{N\to\infty}\IE(\chi_{1}\chi_{2})\leq 0$, since $(1-J_{0})>0$. Moreover, $\rho>-\frac{2}{\pi}\arcsin\left(\frac{1}{M-1}\right) $, since
\begin{align*}
   \tau~&=~\frac{\bar{J}}{1-J_{0}-(M-2)\bar{J}}~
=~\frac{\bar{J}}{(1-J_{0}+\bar{J})-(M-1)\bar{J}}\\
&>~\frac{\bar{J}}{-(M-1)\bar{J}}~=~-\frac{1}{M-1}.
\end{align*}
Proposition \ref{prop:rho} then follows from Lemma \ref{lem:sin_id}.

\subsection{\label{subsec:Proof-thm-identical_pos_corr}Proof of Theorem \ref{thm:identical_pos_corr} and Theorem \ref{thm:hostile_world}}

By Theorem \ref{thm:clt}, the covariance matrix $C$ of the normalised
voting margins is $\left(\mathbf{I}-\bJ\right)^{-1}$. We invert the matrix $\bI-\bJ$ (see Lemma \ref{lem:Ainv}) and
obtain
\[
C=\left(C_{\lambda\nu}\right)_{\lambda,\nu=1,\ldots,M}=\frac{1}{D_{I-J}}\cdot\begin{cases}
1-J_{0}-\left(M-2\right)\bar{J}, & \lambda=\nu,\\
\bar{J}, & \lambda\neq\nu,
\end{cases}
\]
where the constant $D_{I-J}>0$ will not play an important
role in the calculation of the optimal weights. We also note that
all diagonal entries are equal and so are all off-diagonal entries.

Next, we calculate the entries of the linear equation system \eqref{eq:LES}
using the results from Theorem \ref{prop:A_b_high}. The entries of
matrix $A$ have the form
\[
\left(A\right)_{\lambda\nu}=\begin{cases}
1, & \lambda=\nu,\\
\frac{2}{\pi}\arcsin\left(\frac{\bar{J}}{1-J_{0}-\left(M-2\right)\bar{J}}\right), & \lambda\neq\nu.
\end{cases}
\]
We set $\rho$ equal to the off-diagonal
entries of $A$, $\rho:=\frac{2}{\pi}\arcsin\left(\frac{\bar{J}}{1-J_{0}-\left(M-2\right)\bar{J}}\right)$. The entries of the vector $b$ are given by
\begin{align*}
b_{\lambda} & =\IE\left(\chi_{\lambda}\frac{S}{\sqrt{N}}\right)\approx\IE\left(\chi_{\lambda}\frac{S_{\lambda}}{\sqrt{N_{\lambda}}}\right)\sqrt{\alpha_{\lambda}}+\sum_{\nu\neq\lambda}\IE\left(\chi_{\lambda}\frac{S_{\nu}}{\sqrt{N_{\nu}}}\right)\sqrt{\alpha_{\nu}}\\
 & =\sqrt{\frac{2}{\pi C_{\lambda\lambda}}}\left[C_{\lambda\lambda}\sqrt{\alpha_{\lambda}}+\sum_{\nu\neq\lambda}C_{\nu\lambda}\sqrt{\alpha_{\nu}}\right]\propto\left(1-J_{0}-\left(M-1\right)\bar{J}\right)\sqrt{\alpha_{\lambda}}+\bar{J}\eta,
\end{align*}
where $\eta$ is as defined in Section \ref{subsec:Positive-Coupling-High}.
We dropped the multiplicative constant which is identical for all
$\lambda$.

We invert the matrix $A$,
\[
\left(A^{-1}\right)_{\lambda\nu}=\frac{1}{D_{A}}\cdot\begin{cases}
1+\left(M-2\right)a, & \lambda=\nu,\\
-a, & \lambda\neq\nu,
\end{cases}
\]
and proceed to calculate the optimal weights. Dropping common multiplicative
constants and simplifying,
\begin{align*}
w_{\lambda} & =\left(A^{-1}b\right)_{\lambda}\propto\left(1+\left(M-2\right)\rho\right)b_{\lambda}-\rho\sum_{\nu\neq\lambda}b_{\nu}\\
 & =\left(1+\left(M-1\right)\rho\right)\left(1-J_{0}-\left(M-1\right)\bar{J}\right)\sqrt{\alpha_{\lambda}}
+\left[\left(1+\left(M-2\right)a\right)\bar{J}-\rho\left(1-J_{0}\right)\right]\eta.
\end{align*}

The positivity of $D_{1}$ follows immediately, since the second factor
$1-J_{0}-\left(M-1\right)\bar{J}$ is positive in the weak coupling
regime as shown previously. As for $D_{2}$, the inequality $D_{2}\geq0$
is equivalent to
\[
\rho\leq\frac{\bar{J}}{1-J_{0}-\left(M-2\right)\bar{J}},
\]
and thus the claim follows from the following

\begin{lem}
\label{lem:sin_id}For all $x\in\left[0,1\right]$, the inequality
$x\leq\sin\left(\frac{\pi}{2}x\right)$ is satisfied. It holds with
equality if and only if $x\in\left\{ 0,1\right\} $.
\end{lem}

\begin{proof}
Set
\[
f(x):=\sin\left(\frac{\pi}{2}x\right)-x.
\]
The function $f$ has the values $f(0)=f(1)=0$,
$f''(x)=-\frac{\pi^{2}}{4}\sin\left(\frac{\pi}{2}x\right)$. So $f$ is concave on $[0,1] $ and strictly concave on $(0,1) $.
As a consequence, $f(x)\geq 0$ holds on $[0,1] $.
\end{proof}

\subsection{\label{subsec:Proof-thm_two_cluster}Proof of Theorem \ref{thm:two_cluster_weights}}

The proof of this theorem proceeds along the same lines as that of
Theorem \ref{thm:identical_pos_corr}. The main difference is the
inversion of the coupling matrix $\mathbf{I}-\mathbf{J}$. In the following, let $I$ stand for the identity matrix whose dimensions should be clear from the context. The block matrix form allows
us to calculate its inverse using the Schur complement formula
\begin{align*}
\left(\mathbf{I}-\mathbf{J}\right)^{-1} & =\left(\begin{array}{cc}
I-J^{1} & -B\\
-B^{T} & I-J^{2}
\end{array}\right)^{-1}\\
 & =\left(\begin{array}{cc}
\left(I-J^{1}-B\left(I-J^{2}\right)^{-1}B^{T}\right)^{-1} & 0\\
0 & \left(I-J^{2}-B^{T}\left(I-J^{1}\right)^{-1}B\right)^{-1}
\end{array}\right)\\
& \quad \; \cdot \left(\begin{array}{cc}
I & B\left(J^{2}\right)^{-1}\\
B^{T}\left(J^{1}\right)^{-1} & I
\end{array}\right)
\end{align*}
since the matrices $I-J^{i}$ are both invertible in the weak coupling
regime. After lengthy but straightforward calculations, we obtain $C:=(\bI-\bJ)^{-1}=$
\[
C_{\lambda\nu}=\frac{1}{D_{I-J}}\cdot\begin{cases}
1-J_{0}-\left(M-2\right)\bar{J}, & \lambda=\nu,\\
\bar{J}, & \lambda,\nu\in C_{i},i=1,2,\\
-\bar{J}, & \lambda\in C_{i},\nu\notin C_{i},i=1,2.
\end{cases}
\]

Afterwards, the calculation of the optimal weights proceeds along
the same lines as in previous proofs, carefully keeping track of
the signs of the terms.

\subsection{Basic Results on Strong Coupling}\label{strong_minimal_F}

In the strong coupling regime, the matrix $A=\lim_{N\to\infty}\left(\IE\left(\chi_{\nu}\chi_{\lambda}\right)\right)_{\nu,\lambda=1,\ldots,M}$ is singular so that
the limit of the linear system \eqref{eq:LESM} does not have a unique solution. To compute the matrix $A$, we need to find the minima of the function $F$ defined in \eqref{eq:F}. These minima also determine the limiting distribution of the vector of normalised group voting margins (see Theorem \ref{thm:clt}). Since $F$ is continuous and $F(x)\to\infty$ as $\|x\|\to\infty $, the function $F$ does have minima. In the
weak coupling regime, the origin is the unique minimum of $F$. In the strong coupling regime, $0$ is \emph{not} a minimum
of $F$ and all minima come in pairs $x_{0}$ and $-x_{0} $. For the case of {independent} groups of voters, i.e.\! for diagonal $\bJ $, there is a unique minimum of $F$ in each orthant $\mathcal{O}^{\xi}=\{ x\in\IR^{M}\mid x_{\lambda}\,\xi_{\lambda}>0,\, \lambda=1,\ldots,M\}$, where $\xi\in\{ -1,1 \}^{M} $. For \emph{homogeneous} coupling matrices $\bJ=\bJ(J_{0},J_{0}) $, there is a unique pair $x_{0},-x_{0} $ of minima with $x_{0}\in\mathcal{O}^{+}:=\{ x\mid x_{\lambda}>0 \text{ for all }\lambda\}$. In fact, this vector $x_{0}$ has the form
$x_{0}=(m,m,\ldots,m) $, where $m$ satisfies
\begin{equation}
\sum_{\lambda=1}^M \tanh\left(\frac{J_{0}}{\sqrt{\alpha_{\lambda}}}m\right)=m.\label{eq:CW}
\end{equation}
This observation leads to the limit Theorem \ref{thm:clt} (see \cite{KT2021_CWM} for details).

In the general case, it is unknown if and when minima come in \emph{unique} pairs or where they are located.
However, for the classes of coupling matrices introduced in Section \ref{sec:weak}, we have a partial answer on the location of the minima.

\begin{prop}
\label{prop:orthants}Let the model  be in the strong coupling regime.
Then the minima of the function $F$ defined in \eqref{eq:F} are
located in specific orthants of $\IR^{M}$:
\begin{enumerate}
\item In the `friendly world' scenario presented in Section \ref{subsec:Positive-Coupling-High}, i.e.\! if $\jj>0 $,
the global minima are found in the positive and the negative orthant, i.e.\! in the sets $\mathcal{O}^{+}=\{ x\mid x_{\lambda}>0 \text{ for all }\lambda \}$ and $\mathcal{O}^{-}=\{ x\mid x_{\lambda}< 0 \text{ for all }\lambda \}$.

\item In the `hostile world' scenario defined in Section \ref{subsec:hostile}
with equal group sizes, global minima are located in $\left(\begin{array}{c}
M\\
M/2
\end{array}\right)$ orthants if $M$ is even and $\left(\begin{array}{c}
M\\
\left(M+1\right)/2
\end{array}\right)$ if $M$ is odd. The orthants in question are those where half the
coordinates (or $\left(M\pm1\right)/2$ if $M$ is odd) are positive
and the other half (or $\left(M\mp1\right)/2$ if $M$ is odd) are
negative.
\item In the  case of the `split world' scenario defined in Section \ref{subsec:Two-Clusters-high}, the global minima are found in the two orthants with positive coordinates
for $\lambda\leq M_{1}$ and negative entries for $\lambda>M_{1}$
and vice versa.
\end{enumerate}
\end{prop}
\begin{proof}
We only show the first of these results. The others can be proved
analogously.

We first show that only the positive and the negative orthant can
contain any global minima. For the friendly world scenario considered
in Section \ref{subsec:Positive-Coupling-Low}, the expression $\frac{1}{2}y^{T}\sqrt{\alpha}J^{-1}\sqrt{\alpha}y$, $y\in \IR^M$,
can be written as
\[
\left(J_{0}+\left(M-1\right)\bar{J}\right)\sum_{\lambda}\alpha_{\lambda}y_{\lambda}^{2}-\bar{J}
\left(\sum_{\lambda}\sqrt{\alpha_{\lambda}}y_{\lambda}\right)^{2}.
\]
We can thus write the function $F$ defined in \eqref{eq:F} as the sum of two auxiliary functions
$F\left(y\right)=f\left(y\right)+g\left(y\right)$, with
\begin{align*}
f\left(y\right) & :=\left(J_{0}+\left(M-1\right)\bar{J}\right)\sum_{\lambda}\alpha_{\lambda}y_{\lambda}^{2}-\sum_{\lambda}\alpha_{\lambda}\ln\cosh y_{\lambda},\\
g\left(y\right) & :=-\bar{J}\left(\sum_{\lambda}\sqrt{\alpha_{\lambda}}y_{\lambda}\right)^{2}.
\end{align*}
Note that $f$ is independent of the sign of the coordinates of the
argument $y$. More precisely, for any $y$ and any sign vectors $s,s'\in\left\{ -1,1\right\} ^{M}$,
we have
\[
f\left(s\circ y\right)=f\left(s'\circ y\right),
\]
where the symbol `$x\circ y$' stands for coordinatewise multiplication
of the two $M$-vectors $x$ and $y$. Therefore, when comparing values
of $F$ between different orthants, we have to look at the function
$g$. For a fixed $y$ with non-negative coordinates, the minimal
point $s\circ y$, $s\in\left\{ -1,1\right\} ^{M}$, of $g$ is the
one which maximises $\left|\sum_{\lambda}\sqrt{\alpha_{\lambda}}s_{\lambda}y_{\lambda}\right|$.
There are two such $s$, namely $s=\left(1,\ldots,1\right)$ and $s=\left(-1,\ldots,-1\right)$.
This shows that, for $s'\in\left\{ -1,1\right\} ^{M}$ with mixed coordinates,
we have $F\left(y\right)=F\left(-y\right)\geq F\left(s'\circ y\right)$.
For any $y$ with strictly positive coordinates, we even have $F\left(y\right)=F\left(-y\right)>F\left(s'\circ y\right)$.
Thus, if a global minimum is located in the interior of an orthant,
said orthant can only be the positive or the negative one.

Next, we prove that the global minima have to lie in the interior
of the positive and negative orthants, i.e.\! there cannot be any
coordinates with the value 0. To obtain a contradiction, assume that
$y^{*}$ is a global minimum located in the positive orthant and its
coordinate $y_{\nu}^{*}$ is 0. We calculate the partial derivative
of $F$ with respect to $y_{\nu}$:
\[
\frac{\partial F}{\partial y_{\nu}}\left(y\right)=2\alpha_{\nu}y_{\nu}\left(J_{0}+\left(M-1\right)\bar{J}\right)-\alpha_{\nu}\tanh y_{\nu}-2\bar{J}\sqrt{\alpha_{\nu}}\sum_{\lambda}\sqrt{\alpha_{\lambda}}y_{\lambda}.
\]
By assumption, we have
\[
\frac{\partial F}{\partial y_{\nu}}\left(y^{*}\right)=-2\bar{J}\sqrt{\alpha_{\nu}}\sum_{\lambda}\sqrt{\alpha_{\lambda}}y_{\lambda}^{*}.
\]
Since the origin is not a minimum of $F$ in the strong coupling
regime, there must be at least one coordinate with $y_{\lambda}^{*}>0$,
and hence $\frac{\partial F}{\partial y_{\nu}}\left(y^{*}\right)<0$
holds. This implies that moving from $y^{*}$ in the positive direction
of the coordinate $\nu$ (into the interior of the positive orthant)
\emph{decreases} the value of $F$. This contradicts the assumption
that $y^{*}$ is a global minimum. Similarly, we can show the claim
that there cannot be coordinates with the value 0 in global minima
in the negative orthant.
\end{proof}

\subsection{\label{subsec:Limit-Theorems} Proof of Theorem \ref{thm:rho1} }

With $\rho=\lim_{N\to\infty} \IE(\chi_{1}\chi_{2}) $ the matrix $A$ has entries $1$ on the diagonal and $\rho$ away from the diagonal.

Recall that in \eqref{prod_measure} we defined for any $t\in \IR$ $P_t$ as the probability measure on $\{-1,1\}$ with $P_t(1)=(1+\tanh t)/2$ and $P_{t}^{\otimes n}$ as the $n$-fold product measure of $P_t$. According to \eqref{eq:hubstra} and using the law of large numbers for the product measures $\prod_{\nu=1}^M P_{x_\nu}^{\otimes N_{\nu}}$, $x\in \IR^M$, the correlation $\rho $ can be written as $\rho\approx\frac{Z_{2}}{Z}$, where
\begin{align}
   Z~&=~\int_{\IR^{M}}\,e^{-NF(x)}\,\textup{d}x,\label{eq:Z}\\
Z_{2}~&=~\int_{\IR^{M}}\,\chi(x_{1})\,\chi(x_{2})\;e^{-NF(x)}\textup{d}x\,. \label{eq:Z2}
\end{align}
Moreover, by adding a constant to $F$, we may assume that $ \min F~=~0$.

For $R$ large enough, we have $F(x)\geq c \|x\|^{2}$ for all $x\not\in B_{R} $, the ball of radius $R$ around the origin.

By Proposition \ref{prop:orthants}, the minima of $F$ lie in the set
\begin{align*}
D_{\delta}:=\{ x\mid x_{\lambda}>\delta \text{ for all } \lambda \} \;\cup\;\{ x\mid x_{\lambda}<-\delta \text{ for all } \lambda \}.
\end{align*}
We can choose $\delta>0$ such that $F(x)\geq \varepsilon>0$ on $\IR^{M}\setminus D_{\delta}$.

We split the integrals \eqref{eq:Z} and \eqref{eq:Z2} in integrals over $D_{\delta}$, over $B_{R}\setminus D_{\delta}$, and over $\IR^{M}\setminus B_{R}$. Then
\begin{align*}
   &I_{1}~:=~\int_{\IR^{M}\setminus B_{R}} e^{-NF(x)}\,\textup{d}x~\leq~\int_{\IR^{M}\setminus B_{R}} e^{-c_{1}N|x|^{2}}~\leq~e^{-c_{2}NR^{2}},\\
   &I_{2}~:=~\int_{B_{R}\setminus D_{\delta}} e^{-NF(x)}\,\textup{d}x~\leq~R^{M}\,e^{-N\varepsilon},
\end{align*}
so these two expression go to zero exponentially fast in $N$.

Suppose $x_{0}$ is a minimum of $F$, so $x_{0}\in D_{\delta}$. Since the norm of the gradient of $F$, $\| \nabla F \| $, is locally bounded, there is a $\gamma>0$ and a constant $c_{2} $ such that
$F(x_{0}+x)\leq c_{2} \| x \|$ for $\|x\|\leq \gamma $.

Thus, we have
\begin{align*}
  I_{3}~:&=~\int_{D_{\delta}} e^{-NF(x)}\,\textup{d}x~\geq~\int_{B_{\gamma}(x_{0})} e^{-NF(x)}\,\textup{d}x\\
~&\geq~\int_{B_{\gamma}} e^{-c N|x|}\,\textup{d}x~\geq~\int_{B_{\gamma/N}} e^{-c N|x|)}~\geq c_{3}\,e^{-c}\;\frac{\gamma^{M}}{N^{M}}.
\end{align*}
Hence, $I_{3}$ is the leading term as $N\to\infty$.

A very similar argument shows, that the leading term for $Z_{2}$ is the integral
\begin{align*}
   I_{3}'~:=~\int_{D_{\delta}}\chi(x_{1})\,\chi(x_{2})\; e^{-NF(x)}\,\textup{d}x~=~\int_{D_{\delta}} e^{-NF(x)}\,\textup{d}x~=~I_{3}
\end{align*}
since $\chi(x_{1})\,\chi(x_{2})=1 $ on $D_{\delta}$.

Consequently,
\begin{align*}
  \IE(\chi_{1}\,\chi_{2})~\approx~\frac{Z_{2}}{Z}~\approx~1\qquad\text{as } N\to\infty.
\end{align*}
This proves $\rho=1 $.

Along the same lines, one proves
\begin{align*}
   \IE\left(S_{\nu}\,\chi_{\lambda}\right)~\approx~\IE\left(|S_{\nu}|\right).
\end{align*}
From this, we conclude \eqref{eq:b}.

\subsection{\label{CBM}Collective Bias Model}

We include this section to introduce and very briefly discuss the collective bias model, another multi-group probabilistic voting model applicable to the same types of problems as the MFM analysed in the present article. For a far more thorough discussion, see \cite{KT2021_CBM}.

Consider the same setup of a general population subdivided into several groups defined in Section \ref{sec:Def}. Instead of giving the most general definition of the collective bias model, which can be consulted in Definition 5 of  \cite{KT2021_CBM}, we give a specific example of a collective bias model.

Let $Z$ and $Y_\lambda$, $\lambda = 1,\ldots,M$, be independent random variables uniformly distributed on the interval $[-1/2,1/2]$, and define for each $\lambda = 1,\ldots,M$ the random variable $T_\lambda \coloneq Z + Y_\lambda$. These random variables represent biases which arise in the population and they may have different magnitudes and signs depending on the issue at hand. $Z$ represents a prevalent global bias which exists across group boundaries, and each $Y_\lambda$ represents a group bias specific to group $\lambda$. These two biases may have the same or different signs independently of each other. The sum of these two biases  $T_\lambda$ gives the overall bias prevalent in group $\lambda$. If the sign of $T_\lambda$ is positive, each voter in group $\lambda$ has a higher probability of voting `yes'; if $T_\lambda$ is negative, each voter has a higher probability of voting `no'. Given a realisation of $T_\lambda$, each voter casts their vote independently of everyone else. However, the bias introduces a correlation between the individual votes.

Recall the definition of the probability measures $P_t$ on $\{-1,1\}$ for each $t \in [-1,1]$ given in \eqref{prod_measure} and the product measure $P_{t}^{\otimes n} $ immediately afterwards. Then the voting measure that assigns the probability
\begin{align*}
& \qquad \IP\left(X_{11}=x_{11},\ldots,X_{MN_{M}}=x_{MN_{M}}\right)\\
& =\int_{-1/2}^{1/2}\left(\int_{-1/2}^{1/2}P_{z+y_{1}}^{\otimes N_1} \left(x_{11},\ldots,x_{1N_{1}}\right)\textup{d}y_{1}\cdots\int_{-1/2}^{1/2}P_{z+y_{M}}^{\otimes N_M} \left(x_{M1},\ldots,x_{MN_{M}}\right)\textup{d}y_{M}\right)\textup{d}z
\end{align*}
to each voting configuration $\left(x_{11},\ldots,x_{MN_{M}}\right)\in\left\{ -1,1\right\} ^{N}$ is a collective bias model. This is the model treated in Section 8.1.1 of \cite{KT2021_CBM}.

For this collective bias model, the optimal weights which minimise the democracy deficit are given by the formula
\begin{align*}
w_\lambda = \frac{1}{4} \alpha_\lambda + \frac{1}{4} \frac{1}{M + 2}
\end{align*}
for each group $\lambda$. We see that the optimal weight is composed of the sum of two terms: one term is proportional to the size of the population and the other is constant and the same for all groups. Similar results hold under far more general assumptions than the example presented here. It is a formula for voting weights which is akin to how the Electoral College in the United States of America is composed. It stands in contrast to the optimal weights for the MFM, which does not feature uniquely determined weights with a summand which is proportional to the size of each group and a constant summand.

\end{document}